\documentclass{amsart}

\usepackage{
amsfonts,
latexsym,
amssymb,
amsmath,
amsthm,
enumerate,
verbatim,
mathrsfs,
epigraph,
}

\usepackage{url}

\newcommand{\labbel}{\label}

\newtheorem{theorem}{Theorem}[section]
\newtheorem{lemma}[theorem]{Lemma}

\newtheorem{proposition}[theorem]{Proposition} 
 
\newtheorem{corollary}[theorem]{Corollary} 
 
\newtheorem*{claim}{Claim}

\newtheorem*{theorem*}{Theorem}
\newtheorem*{corollary*}{Corollary}

\theoremstyle{definition}
\newtheorem{definition}[theorem]{Definition}

\newtheorem{problem}[theorem]{Problem}

\theoremstyle{remark}
\newtheorem{remark}[theorem]{Remark}
 
\newtheorem{example}[theorem]{Example}

\newtheorem*{acknowledgement}{Acknowledgement}

\numberwithin{equation}{section}

 \allowdisplaybreaks[1]

\begin{document}
 
\title[Terms witnessing
 congruence modularity]
{On the number of terms witnessing \\
 congruence modularity}

\author{Paolo Lipparini} 
\address{ 
Dipartimento Modulare di Matematica\\Viale della  Ricerca
 Scientifica\\Universit\`a di Roma ``Tor Vergata'' 
\\I-00133 ROME ITALY}
\urladdr{http://www.mat.uniroma2.it/\textasciitilde lipparin}

\keywords{Congruence modular varieties, Day terms, Day modularity spectrum, Gumm terms, 
Tschantz terms, Tschantz spectrum,
congruence identities}

\subjclass[2010]{Primary 08B10}
\thanks{Work performed under the auspices of G.N.S.A.G.A. Work 
partially supported by PRIN 2012 ``Logica, Modelli e Insiemi''.
The author acknowledges the MIUR Department Project awarded to the
Department of Mathematics, University of Rome Tor Vergata, CUP
E83C18000100006.}

\begin{abstract}
We study the validity of congruence 
inclusions of the form
   $ \alpha ( \beta \circ \alpha \gamma \circ \beta \circ  \dotsc \circ 
\alpha \gamma \circ  \beta )
\subseteq \alpha \beta \circ \alpha \gamma \circ \alpha \beta \circ \dots$
in congruence modular varieties, 
with an appropriate number of terms on each side  and
where juxtaposition denotes intersection.
Two different methods using Day and Gumm terms are merged
in order to obtain the so far best bounds.
We introduce and study other related identities, 
possibly involving tolerances and admissible relations.
We also slightly improve
a result by A.\ Day, to the effect that if $n$ is even,
then every variety with $n+2$ J{\'o}nsson terms has $2n+1$
Day terms.
\end{abstract} 

\maketitle

(Note added April 2020. Many results
from Section \ref{gumm} here
have been improved in  arXiv:2001.06705
or can be improved using similar techniques.

Some results here are shown to be best possible,
or close to be best possible in arXiv:1902.05995.

One one hand, the present manuscript surely deserves to be updated
according to the above comments;
on the other hand, we believe that many
observations, methods and problems presented here still maintain their
interest. Indeed, some of the results
presented here have been the starting point for arXiv:1902.05995,
which incorporates and extends  parts of  Section \ref{probss}
here.

 Except for the present note, 
the remaining parts of the manuscript
are unchanged with respect to  v2.)

\section{Introduction} \labbel{in} 

\subsection{Summary of the results} \labbel{summ} 
We first sum up the main results of the present work.
The reader unfamiliar with notations and definitions
will find all the necessary explanations below.
We show that if  $\kappa=2r$ is even and some variety $\mathcal V$ 
is congruence modular, as witnessed by 
$k+1$ Day terms, then, for every $q \geq 1$,  
$\mathcal V$ satisfies the congruence identity 
\begin{equation}\labbel{conid}  
 \alpha ( \beta \circ \alpha \gamma \circ \beta \circ  \dotsc \circ 
\alpha \gamma \circ  \beta )
\subseteq \alpha \beta \circ \alpha \gamma \circ \alpha \beta \circ \dots,
 \end{equation}     
with $2 ^{q+1}-1 $ factors inside the parenthesis on the left-hand 
side and $ 2r^q$ factors on the right-hand side.
See Theorem \ref{thm}.
Juxtaposition denotes intersection;
the number of factors is computed 
with respect to the operation of composition;
 terms as $\alpha \gamma $
are counted as a single factor.

If $\mathcal V$ has $n+2$ Gumm terms
and $q \geq 2$, then the  identity \eqref{conid} holds with    
$2^q -1 $ factors inside the parenthesis 
 and $  (2^{q+1}-2q-2)n+3$ factors on the right.
See Theorem \ref{qkmod}. 
So far, in general, the best evaluation is obtained by
combining the two methods. See Theorem \ref{comb}.  It is an open problem 
whether there is a better way. Other open problems related
to similar identities are discussed. The results might shed new light
to the problem of the relationship
 between the number of Day terms and the number of Gumm terms
for a congruence modular variety. 
We comment a bit about the nature of Gumm terms, noticing that 
they can be interpreted also as defective ALVIN distributivity terms. 
See Remark \ref{dj}.

By the way, 
 we use Gumm terms in order to give bounds of the form
$     \alpha (\beta \circ   \gamma  \circ  \beta \circ \dots )
\subseteq
\alpha ( \gamma \circ \beta ) 
\circ \alpha \gamma  \circ \alpha \beta \circ  \alpha \gamma \circ \dots  $,
with appropriate numbers of factors. This extends ideas of
H.-P.\ Gumm and
S.\ Tschantz and relies deeply
on identities involving also admissible relations. See Section \ref{gumm}.
In Section \ref{probss} 
we show that if $n$ is even,
then every variety with $n+2$ J{\'o}nsson terms has $2n+1$
Day terms. This slightly improves
a classical result by A.\ Day \cite{D}.

In Section \ref{probsss}
we introduce identities
of the form
\begin{equation*}  
\alpha ( \beta \circ \alpha ( \gamma \circ \alpha ( \beta \circ \dots
\alpha ( \gamma    \circ
 \alpha ( \beta \circ \alpha \gamma   \circ  \beta )
 \circ \gamma ) \ldots
 \circ \beta )\circ \gamma) \circ \beta ) \subseteq
\alpha \beta \circ  \alpha \gamma  \circ \alpha \beta \circ \dots
\end{equation*}
Some arguments there show that such identities are quite natural.
For any given congruence modular variety,
the problem naturally arise
of evaluating the minimal  value for the number of 
factors on the right, in the various identities we have considered.
We call this the \emph{generalized Day spectrum
problem} and  briefly consider it in Section \ref{probsss}.  
The problem is still largely open.

\subsection{Background} \labbel{backg} 
We now explain in some detail the background and our
main aims. By a celebrated theorem by 
A.\ Day \cite{D}, a variety $\mathcal V$ is congruence modular
if and only if there is an integer $k >0$ such that $\mathcal V$ satisfies 
\begin{equation}\labbel{1}    
 \alpha ( \beta \circ \alpha \gamma \circ \beta  ) \subseteq 
\alpha \beta \circ \alpha \gamma \circ \alpha \beta \dots \quad
  \text{($k-1$ occurrences of $\circ$ on the right)},
  \end{equation}
 where 
$\alpha$, $\beta$ and $\gamma$  vary among congruences
of some algebra in $\mathcal V$ and
 juxtaposition denotes intersection.
When we say that some variety $\mathcal V$ \emph{satisfies
an identity}, we mean that the identity is satisfied in all algebras in $\mathcal V$.
Formally, for every algebra $\mathbf A \in \mathcal V$,
 the identity, when interpreted in the
standard way,  has to be satisfied in
the structure of all reflexive and admissible relations  of $\mathbf A$.
All the binary relations in this paper 
are assumed to be reflexive, hence we shall sometimes
say \emph{admissible} in place of reflexive and admissible. 
As we mentioned, $\alpha, \beta , \dots$ are intended
to  vary among congruences;  other kinds 
of variables shall be allowed; for example,
we shall use the letters $R, S, \dots $ to denote variables for 
reflexive and admissible relations.
The above conventions shall be in charge throughout the paper.

A variety satisfying the identity \eqref{1} is said to be \emph{$k$-modular}.  

Day's actual statement concerns the existence of certain terms;
for varieties, this is equivalent to the satisfaction of the above congruence inclusion.
In this sense, 
a variety satisfying the identity \eqref{1} is said to have
\emph{$k+1$ Day terms}. See Section \ref{intr}, in particular,  
Proposition \ref{mod3} and Theorem \ref{day} 
below for further details. 
We notice that two terms in Day's conditions are projections,
hence the number of nontrivial terms 
 in the
condition for   $k$-modularity is $k-1$.
Stating results in the form of an inclusion is 
intuitively clearer and
 notationally simpler,
though in certain proofs the terms are  what is really used.
Cf.\ Tschantz \cite{T} and below.

We now need a bit more notation.
For $m \geq 1$, let 
$ \beta \circ _m \gamma $  denote
$ \beta \circ \gamma \circ \beta \dots$
with $m$ factors, that is, 
with $m-1$ occurrences of $\circ$.
In some cases, when we want to mention the last factor  explicitly,
say, when $m$ is odd, 
we shall write,
$\beta \circ \gamma \circ {\stackrel{m}{\dots}} \circ  \beta   $
in place of $ \beta \circ _m \gamma $.
With this notation, \eqref{1} above reads
$ \alpha ( \beta \circ \alpha \gamma \circ \beta  ) \subseteq 
\alpha \beta \circ_k \alpha \gamma$,
or even $ \alpha ( \beta \circ \alpha \gamma 
\circ {\stackrel{3}{\dots}} \circ \beta  ) \subseteq 
\alpha \beta \circ_k \alpha \gamma$. 

The nontrivial point in  the proof in \cite{D} is to show that
if \eqref{1} above holds in the free algebra in $\mathcal V$ generated by
four elements (and thus we have appropriate terms), then  
$ \alpha ( \beta \circ_m \alpha \gamma  ) \subseteq \alpha \beta + \alpha \gamma $
holds, for every $m$.
Actually, it is implicit in the proof that, for every $m$, there is some 
$D  (m) $ which depends only on the $k$ given by  \eqref{1}
(but otherwise not on the variety at hand)
 and   
  such that 
$ \alpha ( \beta \circ_m \alpha \gamma  ) \subseteq 
\alpha \beta \circ_{D(m)} \alpha \gamma $.
In Section \ref{intr}  we evaluate $D(m)$ as
given by Day's proof and then in Section \ref{b}  we show that 
the alternative proof of Day's result  from \cite[Corollary 8]{L}
  provides a better 
bound. Other bounds are obtained in Section \ref{gumm} 
by using a different set of terms for congruence modular varieties,
the terms discovered by H.-P.\ Gumm. 
There we also provide a bound for 
$     \alpha (\beta \circ   \gamma  \circ  \beta \circ \dots )$, relying
on ideas of S.\ Tschantz.
The  methods are combined
and compared
in Section \ref{probs}, where 
we also discuss the (still open) problem
of finding the best possible evaluation for
$D  (m) $. Finally, some connections 
with congruence distributivity and some further problems
are discussed in Sections \ref{probss} and \ref{probsss}. 

The next section is introductory in character and might
be skipped by a reader familiar with congruence modular
varieties.

\section{Introductory remarks} \labbel{intr}   
   
This section contains a few  introductory and historical remarks,
with absolutely no claim at exhaustiveness.
The reader familiar with congruence modular varieties 
might want to go directly to the next section.

Obviously, an algebra $\mathbf A$ is congruence modular 
if and only if it satisfies
$ \alpha ( \beta \circ_m \alpha \gamma  ) \subseteq \alpha \beta + \alpha \gamma $,
for every $m \geq 3$ and all congruences $\alpha$, $\beta$ and $\gamma$
of $\mathbf A$. 
The by now standard algorithm by   A.\ Pixley \cite{P}
and R.\ Wille
\cite{W}  
thus shows that a variety $\mathcal V$ is congruence modular if and only if,
for every $m \geq 3$, there is some $k$ (which depends on the variety) such that    
$\mathcal V$  satisfies the congruence identity 
$ \alpha ( \beta \circ_m \alpha \gamma  ) \subseteq 
\alpha \beta \circ_k \alpha \gamma $.

Of course,
as we mentioned in the introduction, 
 it is well-known that this is not the best possible result, since A.\ Day
\cite{D} 
showed that already
 $ \alpha ( \beta \circ \alpha \gamma  \circ \beta )
 \subseteq \alpha \beta \circ_k \alpha \gamma $, for some $k$, 
implies congruence modularity.  That is, the case $m=3$ 
in the above paragraph is enough.
For the reader familiar with the terminology,
this implies that congruence modularity is actually a Maltsev condition,
rather than a weak Maltsev condition  
(by the way,  Day's result appeared before the general
formulation of the  
algorithm by  Pixley and  Wille).

It follows immediately from
 Day's result that 
if a variety $\mathcal V$ 
 satisfies the congruence identity 
$ \alpha ( \beta \circ_m \alpha \gamma  ) \subseteq 
\alpha \beta \circ_k \alpha \gamma $,
for some $m \geq 3$ and  $k$, then    
$\mathcal V$ is congruence modular.
Indeed,
if $m \geq 3$, then 
$\alpha ( \beta \circ \alpha \gamma  \circ \beta )
\subseteq   \alpha ( \beta \circ_m \alpha \gamma  )  \subseteq
\alpha \beta \circ_k \alpha \gamma $ 
and we are done by 
Day's result.

For $m \geq 3$, let us say that an algebra $\mathbf A$ is
\emph{$(m,k)$-modular}
if 
$ \alpha ( \beta \circ_m \alpha \gamma) \subseteq 
\alpha \beta \circ_k \alpha \gamma$,
for all congruences of $\mathbf A$.
A variety is \emph{$(m,k)$-modular}
if all of its algebras are.
 By the above discussion, a variety $\mathcal V$ is congruence
modular if and only if $\mathcal V$  is 
$(3,k)$-modular, for some $k$,
if and only $\mathcal V$  is  
$(m,k)$-modular, for some  $m \geq 3$  and some $k$,
if and only if, for every  $m \geq 3$, there is some $k$ 
such that $\mathcal V$ is $(m,k)$-modular.

In the present terminology, $(3,k)$-modularity 
is equivalent to  $k$-modularity in the sense of \cite{D}.
A  $k$-modular variety  is usually said 
to \emph{have $k+1$ Day terms} (see Condition (iv)
in Proposition \ref{mod3} below). 

For convenience, we have given the definition
of  $(m,k)$-modularity also in the case when
$m$ is even, but this does not seem to provide
a big  gain in generality, since, say,
$ \alpha ( \beta \circ \gamma \alpha \circ \beta \circ \alpha \gamma  )=
 \alpha ( \beta \circ \gamma \alpha \circ \beta) \circ \alpha \gamma  $.  
In other words, if we define $D_{ \mathcal V} (m) $  
to be the least $k$ such that $\mathcal V$ is  $(m,k)$-modular,
we have $D_{ \mathcal V} (m+1) \leq  D_{ \mathcal V} (m) + 1$,
for $m$ odd.   
 Actually, if $m$ is odd and  $k$ is even, 
then  $(m,k)$-modularity is equivalent to 
$(m+1,k)$-modularity.

As now folklore, Day's arguments can be divided in two steps.
 The first step can be seen, by hindsight,
as a mechanical application of 
the  Pixley and  Wille algorithm.
Compare the perspicuous  discussions in 
Gumm \cite{G} and 
 Tschantz \cite{T}.

\begin{proposition} \labbel{mod3}
For a variety $\mathcal V$, the following conditions are equivalent. 
\begin{enumerate}[(i)] 
  \item 
$\mathcal V$ is $k$-modular, that is,  $(3,k)$-modular, that is,
$ \alpha ( \beta \circ \alpha \gamma \circ \beta ) \subseteq 
\alpha \beta \circ_k \alpha \gamma$ holds in $\mathcal V$.
\item
The free algebra in $\mathcal V$ over four
generators is $(3,k)$-modular.
\item
Suppose that
 $\mathbf A$ is the free algebra in $\mathcal V$ over
the four generators $a,b,c,d$ and
$\alpha$, $\beta$ and $\gamma$ 
are, respectively, the smallest congruences
of $\mathbf A$ containing the pairs
  $ (a,d) $, $(b,c)$,
the pairs   $ (a,b) $, $  (c,d)$
and the pair  $ (b,c)$.
Then $(a,d) \in \alpha \beta \circ_k \alpha \gamma $. 
\item
$\mathcal V$ has $4$-ary  terms 
$d_0, \dots, d_{k}$ such that the following equations are valid in $\mathcal V$:
  \begin{enumerate}    
\item
$x=d_i(x,y,y,x)$, for every $i$;  
\item 
$x=d_0(x,y,z,w)$;
\item
$d_i(x,x,w,w)=d_{i+1}(x,x,w,w)$, for $i$ even;
\item
$d_i(x,y,y,w)=d_{i+1}(x,y,y,w)$, for $i$ odd, and
\item 
$d_{k}(x,y,z,w)=w$.
  \end{enumerate}  
  \end{enumerate} 
 \end{proposition}  

\begin{proof}
By now, standard. We shall only use (i) $\Leftrightarrow $  (iv),
which is implicit in \cite{D}. 
The proof of the full result, again, is implicit
in \cite{D} and can be extracted, for example,  from the proof of 
\cite[Theorem 2.4]{J}. 
\end{proof}

\begin{theorem} \labbel{day}
(Day's Theorem)
A variety $\mathcal V$  is congruence modular if and only if, for some $k$,  
$\mathcal V$  satisfies one (hence all) of the conditions in 
Proposition \ref{mod3}.
 \end{theorem}

The proof of Theorem \ref{day} uses the terms given by 
Proposition \ref{mod3}(iv).
Day's argument essentially
proves  a tolerance identity. 
If an algebra $\mathbf A$ has terms satisfying 
Proposition \ref{mod3}(iv), then $\mathbf A$  satisfies 
\begin{equation}\labbel{ed}    
\alpha ( \Delta  \circ \alpha \gamma  \circ \beta ) \subseteq 
 (\alpha \Delta \circ \alpha \Delta ) \circ_ \kappa \alpha \gamma  
  \end{equation}
where, as usual, $\alpha$ and $\gamma$ are congruences,
but $\Delta$  is only assumed to be a tolerance
containing $\beta$.
Recall that a \emph{tolerance} is a reflexive and symmetrical 
admissible relation.
Now if we set $\Delta_h = 
 \beta \circ \alpha \gamma \circ {\stackrel{2h+1}{\dots}}
\circ \beta  $,
then  $\Delta_h$ is surely 
a tolerance, and  moreover
 $\Delta_{h+1} =  \Delta_{h} \circ \beta \circ \alpha \gamma  $,
hence, using identity \eqref{ed} and by an easy induction,
it can be proved that, for every $h$, we have 
$ \alpha ( \beta \circ \alpha \gamma \circ {\stackrel{2h+1}{\dots}}
\circ \beta ) \subseteq \alpha \beta + \alpha \gamma  $,
hence Theorem \ref{day} follows.    
Actually, the induction shows that, say for $k$ even, we have 
$\alpha ( \beta \circ \alpha \gamma \circ {\stackrel{2h+1}{\dots}}
\circ \beta ) \subseteq \alpha \beta \circ_p \alpha \gamma $,
with $p= k ^{h} $.
In other words, for $k$ even, $(3,k)$-modularity 
implies $(2h+1, k^h)$-modularity, for every $h$.

Day's argument can be modified
in order to improve \eqref{ed} to 
\begin{equation*}
\alpha ( \Delta  \circ \alpha \gamma  \circ \Delta) \subseteq 
 (\alpha \Delta \circ \alpha \Delta ) \circ_ \kappa \alpha \gamma  
  \end{equation*}
for all congruences $\alpha$ and $\gamma$ and
tolerance $\Delta$.
Essentially, the above identity is a reformulation of 
Gumm's Shifting Principle \cite{G}.
Again by induction, we get that, for $k$ even, $(3,k)$-modularity 
implies $(2 ^{q+1}-1 , k^q)$-modularity,
for every $q$.

Some further small improvements are possible analyzing Day's proof, for example
one can get 
\begin{equation}\labbel{eddd}    
\alpha ( \Delta  \circ \alpha \gamma  \circ \Delta) \subseteq 
 \alpha \Delta  \circ (\alpha \gamma  \circ _{k-1} (\alpha \Delta \circ \alpha \Delta)) 
\end{equation}
that is, one can ``save'' an $\alpha \Delta$ in the first place,
and symmetrically this can be done at the end,
if $k$ is odd.  

However, we shall show in the next section that
the identity \eqref{eddd} can be improved further, by  
using results from  \cite{L}.

\section{An improved bound} \labbel{b} 

Recall that, for $m \geq 3$, we say that a variety $\mathcal V$ is
\emph{$(m,k)$-modular}
if the congruence inclusion 
$ \alpha ( \beta \circ_m \alpha \gamma ) \subseteq 
\alpha \beta \circ_k \alpha \gamma$
holds in $\mathcal V$.
Thus $(3,k)$-modularity means having $k+1$ Day terms or,
which is the same, being $k$-modular.
Moreover, we defined 
 $D_{ \mathcal V} (m) $  
as the least $k$ such that $\mathcal V$ is  $(m,k)$-modular.
It follows from the comments in the previous section that 
$D_{ \mathcal V} (m) $  is defined for every $m \geq 3$
and every congruence modular variety $\mathcal V$.

We shall use \cite[Theorem 3 (i) $\Rightarrow $   (iii)]{L} in order
to improve the 
results mentioned in the previous section.
Essentially, we evaluate the number of terms
given by the alternative proof of Day's theorem
presented in \cite[Corollary 8]{L}.
If $R$ is a binary relation on a set $A$,
we let 
$R ^\smallsmile $ denote the \emph{converse}
of $R$, that is,   
$R^\smallsmile = \{ (b,a) \mid a, b \in A, \text{ and } a \mathrel R b   \} $. 
Notice that $R^\smallsmile $ has been denoted by $R^-$ in \cite{L}.

\begin{proposition} \labbel{nt}
If $\mathcal V$ is a $k$-modular variety, 
 that is, $\mathcal V$
is $(3,k)$-modular,  then    
$\mathcal V$ satisfies
\begin{equation}\labbel{nte}    
\alpha ( \Delta  \circ \alpha \gamma  \circ \Delta) \subseteq 
 \alpha \Delta  \circ_ \kappa \alpha \gamma,  
  \end{equation}
for all congruences $\alpha$ and $\beta$ 
on some algebra $\mathbf A$ in $\mathcal V$
and every tolerance  $\Delta$  on $\mathbf A$ such that  there exists
an admissible relation $R$ on $\mathbf A$ for which 
$\Delta = R \circ R^\smallsmile$.
\end{proposition} 

\begin{proof} 
In view of Proposition \ref{mod3},
 this is a special case of
\cite[Theorem 3 (i) $\Rightarrow $  (iii)]{L}.
For the reader's convenience, we present the  explicit details in 
this particular case. 

By Proposition \ref{mod3} we have terms 
as given in \ref{mod3}(iv).   
If $\mathbf A \in \mathcal V$,
$a, d \in A$ and $(a,d) \in \alpha ( \Delta  \circ \alpha \gamma  \circ \Delta) $,
then $a \mathrel \alpha   d$ and  there are elements $b,c \in A$ such that 
$a \mathrel \Delta  b \mathrel { \alpha \gamma} c  \mathrel \Delta d$. 
Since $\Delta = R \circ R^\smallsmile$, we have further elements
$b'$ and $c'$ such that   
$a \mathrel R b' \mathrel {R^\smallsmile}  b \mathrel { \alpha \gamma} 
c  \mathrel R c' \mathrel {R^\smallsmile}  d$,
thus 
$b \mathrel R b'$ and  
$d \mathrel R c'$.

Using the terms 
 given by Proposition \ref{mod3}(iv),
we have  
  \begin{enumerate}    
\item
$d_i(a,b,c,d)  \mathrel \alpha  d_i(a,b,b,a)= a =d_{i+1}(a,b,b,a)
  \mathrel \alpha d_{i+1}(a,b,c,d)   $, for every $i$; 
\item  
$a=d_0(a,b,c,d)$; 
\item
$d_i(a,b,c,d) \mathrel R d_i(b',b', c',c') = 
d_{i+1}(b',b', c',c') \mathrel {R^\smallsmile} 
 d_{i+1}(a,b, c,d)$, for $i$ even; 
\item
$d_i(a,b,c,d) \mathrel \gamma  d_i(a,b, b,d) = 
d_{i+1}(a,b, b,d) \mathrel \gamma 
 d_{i+1}(a,b, c,d)$, for $i$ odd,  and
\item 
$ d_{k}(a,b, c,d) = d$,
\end{enumerate}
thus the elements  $d_i(a,b,c,d) $,
for $i=0, \dots, k$,  
witness the desired inclusion,
noticing that item (3) implies 
$d_i(a,b,c,d) \mathrel \Delta  
 d_{i+1}(a,b, c,d)$, for $i$ even. 
\end{proof}

\begin{remark} \labbel{czh} 
Recall from \cite{L}
that a tolerance $\Theta$ is \emph{representable}
  if $\Theta= R \circ R^\smallsmile$,
 for some admissible relation $R$.
In this terminology,  the assumption on $\Delta$ 
in the statement of \ref{nt} 
asserts exactly that $\Delta$ is representable. 
It follows from  \cite[Theorem 3 (i) $\Rightarrow $  (iii)]{L}
that the identity  \eqref{nte} in  
Proposition \ref{nt}
holds just under the assumption  that 
$\alpha$, $\Delta$ and $\gamma$ are representable tolerances,
with no need to assume that $\alpha$ and $\gamma$ are congruences.

Actually, a finer result holds! Arguing as in \cite{CH}, 
we can relax the assumption that $\alpha$ is a congruence in Proposition \ref{nt} 
to $\alpha$ being \emph{any} tolerance 
(not necessarily representable).
The same argument has been applied several times in 
\cite{jds,ricmc,uar,ricm,B,ia}.

Even more generally, 
we have shown in \cite{nest} 
that the assumption of representability 
can be weakened to nest-representability, as introduced in 
\cite{nest}.
The basic ideas of the definition 
and of the arguments
are presented here in a special case 
in  the Claim contained 
in the proof of Proposition \ref{dst} below.
\end{remark}   

We let $R^h$
be a shorthand for 
$R  \circ _{h} R $,
that is, the relational composition of 
$h$ factors, each factor being equal to $R$.   

We now consider consequences of 
$k$-modularity for $k$ even, say $k=2r$.  

\begin{theorem} \labbel{thm}
For  every $r, q \geq 1$, $(3,2r)$-modularity, that is,
 $2r$-modularity,  
implies $(2 ^{q+1}-1 ,2 r^q )$-modularity.
 \end{theorem}

 \begin{proof} 
By induction on $q \geq 1$.

The induction basis $q=1$ is exactly $(3,2r)$-modularity.

Suppose that the theorem   holds
for some given $q \geq 1$.
We have  
\begin{multline*} 
\alpha (\beta \circ \alpha \gamma \circ {\stackrel{2 ^{q+2}-1}{\dots\dots}}
 \circ \beta) = \\  
\alpha \left( 
(\beta \circ \alpha \gamma \circ {\stackrel{2 ^{q+1}-1}{\dots\dots}} \circ \beta)
\circ \alpha \gamma \circ
(\beta \circ \alpha \gamma \circ {\stackrel{2 ^{q+1}-1}{\dots\dots}} \circ \beta) 
\right )
 \subseteq^ { \text{Prop.\ } \ref{nt}} \\ 
\alpha 
(\beta \circ \alpha \gamma \circ {\stackrel{2 ^{q+1}-1}{\dots\dots}} \circ \beta)
\circ_{2r} \alpha \gamma 
\subseteq ^{\text{ih}}  
(\alpha \beta \circ \alpha \gamma \circ {\stackrel{2 r^q}{\dots}}
 \circ \alpha \gamma )
\circ_{2r} \alpha \gamma  = 
\\
(\alpha \beta \circ \alpha \gamma \circ {\stackrel{2 r^q}{\dots}}
 \circ \alpha \gamma )^{\frac{2r}{2}}
=
\alpha \beta \circ \alpha \gamma \circ {\stackrel{2 r^{q+1}}{\dots}}
 \circ \alpha \gamma,
\end{multline*}   
since $2r$ is even and  $\alpha \gamma $ is a congruence, thus
 $\alpha \gamma  \circ \alpha \gamma = \alpha \gamma $.
The inclusion marked with ``ih''   follows from the inductive hypothesis
and we are allowed  to apply Proposition \ref{nt} since 
 $ \Delta =\beta \circ \alpha \gamma \circ {\stackrel{2 ^{q+1}-1}{\dots\dots}}
 \circ \beta
= R \circ R^\smallsmile$,
for  
$R = \beta \circ \alpha \gamma \circ {\stackrel{2 ^{q}}{\dots}} \circ \beta
\circ \alpha \gamma $, again since 
 $\alpha \gamma  \circ \alpha \gamma = \alpha \gamma $
and since $q \geq 1$, thus $2^q$ is even. 
Moreover, the notation in the displayed formula is consistent, since 
$2 ^{q+1}-1 $ is odd and
$2r^q$ is even.  
\end{proof}  

The proof of Theorem \ref{thm} applies to a more general context.

\begin{proposition} \labbel{thm2}
For every $r, h>1$  and  $i\geq 0$, we have that 
$(2h-1, 2r)$-modularity
implies $(2h ^{2^{i}}-1 , 2r ^{2^i} )$-modularity.
 \end{proposition} 

 \begin{proof}
It is enough to show  that  
$(2h-1, 2r)$-modularity
implies   
$(2h^2-1, 2r^2)$-modularity;
the full result follows  then immediately by induction.
To prove the above claim, notice that, again by 
\cite[Theorem 3 (i) $\Rightarrow $  (iii)]{L},
$(2h-1, 2r)$-modularity implies the identity 
$\alpha ( \Delta  \circ \alpha \gamma  \circ  {\stackrel{2h-1}{\dots}} \circ \Delta)
 \subseteq 
 \alpha \Delta  \circ_ {2r} \alpha \gamma$, 
for $\Delta$ a representable tolerance.
Then the argument in   \ref{thm} applies
to get 
$(2h^2-1, 2r^2)$-modularity, taking $\Delta'= \Delta  \circ \alpha \gamma  \circ  {\stackrel{2h-1}{\dots}} \circ \Delta$ in place of $\Delta$.
 \end{proof} 

By similar arguments one can obtain consequences
of the combination of 
$(2h-1, k)$-modularity
and of 
$(2h'-1, k')$-modularity.

Of course, analogous results 
can be proved for $k=2r-1$ odd, 
but we have not worked out the details.
At  first sight, it appears that
variations on the above arguments
in the case $ k =2r-1$ odd provide
only  minor improvements with respect to the bounds obtained
for (the then even) $k+1= 2r$.
At least, this is partially  confirmed by considering 
what happens for small values of $k$.

Let $0$ denote the smallest congruence on the algebra under consideration. 

\begin{proposition} \labbel{small}
  \begin{enumerate}[(i)]  
  \item 
If $\mathcal V$ is $3$-modular, then 
$\mathcal V$ is $(m, m)$-modular, for every $m \geq 3$. 
\item 
If $\mathcal V$ is $4$-modular, then 
$\mathcal V$ is $(2^q-1, 2^q)$-modular, 
for every  $q \geq 2$.
  \end{enumerate} 
 \end{proposition} 

\begin{proof}
(i) We first prove the result for $m$ odd. 
The proof is by induction on $m\geq 3$.
The base case is the assumption of $3$-modularity.

Suppose that $m=4r+1$ and that (i) holds for every odd $m'<m$.    
Then, letting $\Delta = \beta \circ \alpha \gamma 
\circ {\stackrel{2r+1}{\dots\dots}}
 \circ \beta  $, we get  
\begin{multline*}   
\alpha (\beta \circ \alpha \gamma \circ {\stackrel{4r+1}{\dots\dots}}
 \circ \beta) 
=
\alpha ( \Delta \circ 0 \circ \Delta )
 \subseteq^ { \text{Prop.\ } \ref{nt}} 
\alpha \Delta  \circ  0 \circ \alpha \Delta  
 \subseteq ^{\text{ih}}  
\\
(\alpha \beta \circ \alpha \gamma 
\circ {\stackrel{2r+1}{\dots\dots}}
 \circ \alpha \beta) \circ  
(\alpha \beta \circ \alpha \gamma 
\circ {\stackrel{2r+1}{\dots\dots}}
 \circ \alpha \beta) 
=
\alpha \beta \circ \alpha \gamma 
\circ {\stackrel{4r+1}{\dots\dots}}
 \circ \alpha \beta.  
\end{multline*} 

Now suppose that $m=4r+3$ and 
 let again $\Delta = \beta \circ \alpha \gamma 
\circ {\stackrel{2r+1}{\dots\dots}}
 \circ \beta  $. Then  
\begin{multline*}   
\alpha (\beta \circ \alpha \gamma \circ {\stackrel{4r+3}{\dots\dots}}
 \circ \beta) 
=
\alpha ( \Delta \circ \alpha \gamma  \circ \Delta )
 \subseteq^ { \text{Prop.\ } \ref{nt}} 
\alpha \Delta  \circ  \alpha \gamma  \circ \alpha \Delta  
 \subseteq ^{\text{ih}}  
\\
(\alpha \beta \circ \alpha \gamma 
\circ {\stackrel{2r+1}{\dots\dots}}
 \circ \alpha \beta) \circ  \alpha \gamma \circ 
(\alpha \beta \circ \alpha \gamma 
\circ {\stackrel{2r+1}{\dots\dots}}
 \circ \alpha \beta) 
=
\alpha \beta \circ \alpha \gamma 
\circ {\stackrel{4r+3}{\dots\dots}}
 \circ \alpha \beta.  
\end{multline*} 

The case $m $ even is immediate from the case 
$m$ odd, since if $m$ is even, then
$\alpha (\beta \circ \alpha \gamma \circ {\stackrel{m}{\dots}}
 \circ \alpha \gamma ) 
=
\alpha (\beta \circ \alpha \gamma \circ {\stackrel{m-1}{\dots}}
 \circ \beta ) \circ \alpha \gamma$. 

(ii) is the particular case $k=4$ of Theorem \ref{thm}
(with $q$ shifted by $1$). 
 \end{proof}    

Notice that, as we mentioned, if $m$ is odd, then
 $(m, m+1)$-modularity implies $(m+1, m+1)$-modularity.
Henceforth Proposition  \ref{small} suggests that
 $4$-modularity implies  
$(m+1, m+1)$-modularity, for every odd $m \geq 3$,
but this is still open.
In any case, by a trivial monotonicity property, we have that,
for every $m$, a $4$-modular variety is 
$(m, 2^q)$-modular, where $q$ is the smallest integer such that 
  $m \leq  2^q$. 
A similar remark applies to most results 
of the present paper.

Apart from the above remarks, Proposition  \ref{small}(i),
within its scope, is the best possible result, as shown by
the variety of lattices.   
Indeed, for  $m \geq 3$,
every  $(m, m-1)$-modular variety is $m-1$-permutable:
just take $\alpha=1$ in the definition of $(m, m-1)$-modularity.  
Hence lattices are not
$(m, m-1)$-modular, while they are 
$3$-modular, e.~g., by Day \cite[Theorem on p.\ 172]{D}.
See the discussion at the beginning of Section \ref{probss}. 

\begin{example} \labbel{bex}  
In \cite{B} we provided examples of $5$-modular varieties  
which are
$(m, \allowbreak 2m-1)$-modular but 
 not 
$(m,2m-2)$-modular, for $m$ odd. 
See Theorem 2.1 and Corollary 2.4 in \cite{B}.
Thus, for example, we have a $5$-modular variety
which is  
$(7,13)$-modular, but 
 not 
$(7,12)$-modular.
On the other hand, 
Theorem \ref{thm}
implies that every $6$-modular variety   
is $(7,18)$-modular.
Arguing as in the proof of  
Proposition \ref{small} (i)
it can be seen that  
every $5$-modular variety   
is $(7,17)$-modular.
While there is  still room for  refinements,
the above examples show that our results are
relatively sharp, at least for small values of 
$m$. 
\end{example}

\section{Employing  Gumm terms} \labbel{gumm} 

Another very interesting characterization of congruence modularity 
has been provided by Gumm \cite{G1}. See also \cite{G}
and Freese and McKenzie \cite{FMK} for the related commutator theory for
congruence modular varieties.

\begin{theorem} \labbel{gummt}
(Gumm \cite{G1})
For a variety $\mathcal V$, the following conditions are equivalent. 
\begin{enumerate}[(i)] 
  \item 
$\mathcal V$ is congruence modular.
\item
For some $n$, $\mathcal V$ has ternary terms 
$p$ and  $j_1, \dots, j_{n+1}$
  such that the following  equations are valid in $\mathcal V$:
  \begin{enumerate}    
\item
$x=j_i(x,y,x)$, for every $i$;  
\item 
$x=p(x,z,z)$;
\item 
$p(x,x,z) = j_{1}(x,x,z)$;
\item
$j_i(x,z,z)=j_{i+1}(x,z,z)$, for $i$ odd, $i \leq n$;
\item
$j_i(x,x,z)=j_{i+1}(x,x,z)$, for $i$ even, $i \leq n$, and
\item
$j_{n+1}(x,y,z) = z$. 
  \end{enumerate}  
  \end{enumerate} 
 \end{theorem} 

We say that \emph{$\mathcal V$ has $n+2$ Gumm terms}
if Condition (ii)  above holds for $\mathcal V$ and  $n$.

Notice that condition (ii)
as presented here is slightly different from 
Gumm's corresponding condition \cite[Theorem 7.4(iv)]{G},
in which the ordering of the terms is exchanged.
  This is not just a matter of symmetry, see, for example,  the comment 
in \cite[p.\ 12]{jds}.
Seemingly, the first formulation of condition \ref{gummt}(ii)
above first appeared in
\cite{LTT} and \cite{T}.   

As already noted by Gumm himself,
conditions like  \ref{gummt}(ii) above ``compose'' the corresponding conditions
for congruence permutability and congruence distributivity 
due to Maltsev \cite{M}  and  J{\'o}nsson \cite{JD}, respectively.
In fact,  Gumm terms
provide the possibility
of merging methods used for
congruence distributivity and congruence permutability,
as  accomplished in several ways by several authors.
See again  \cite{G,FMK} for details.
Lemma \ref{appgumm}  below is another small result in the same vein.
 Notice that 
the existence of $n+2$ Gumm terms is
the Maltsev condition naturally associated
with the congruence identity
$ \alpha ( \beta \circ \gamma ) \subseteq 
 \gamma \circ \beta \circ (\alpha \gamma  \circ_{n} \alpha \beta  ) $,
or even  
$ \alpha ( \beta \circ \gamma ) \subseteq 
\alpha ( \gamma \circ \beta ) \circ (\alpha \gamma  \circ_{n}
 \alpha \beta  ) $.
Here, to be consistent in the extreme case $n=0$,
we  conventionally assume  
$\alpha \gamma  \circ_{0}
 \alpha \beta  =0$; this is another way to see
that having $2$ Gumm terms corresponds to 
congruence permutability.
Before going on, we shall discuss another possible interpretation of Gumm terms.

\begin{remark} \labbel{dj}
 We should remark that Gumm terms can be interpreted still 
in another way,
 besides expressing some condition having  the flavor 
of  ``permutability composed with
distributivity''.  

(i) \emph{The classical interpretation.} 
Before explaining the other possible interpretation
of Gumm terms, let us add some details
about the  above expression  in quotes.
As already noted by Gumm himself, 
if $n=1$ (or, equivalently, if all the $j_i$'s are the trivial projections
onto the third coordinate), then the only nontrivial
term is $p$, which satisfies the Maltsev equations
for congruence permutability. In this sense $p$
represents the permutability part.
On the other hand, 
if $p$ is the trivial projection onto the first coordinate,
then the terms give J{\'o}nsson condition
characterizing congruence distributivity. 
If this is the case, that is, 
 $p$ is the trivial projection onto the first coordinate, 
then it is convenient to relabel $p$ as
$j_0$, thus we have $n+2$ terms,
of which only $n$ are nontrivial.
The resulting condition is usually called \emph{$n+1$-distributivity}.
In this sense, the $j_i$'s 
represents the distributivity part.
More specifically, 
\emph{J{\'o}nsson terms} for $n+1$ distributivity 
are terms 
$j_0, \dots, j_{n+1}$
satisfying 
conditions (a), (d)-(f) above
as well as 
$x= j_0(x,y,z)$.
Notice that, if we rename $p$ as $j_0$,
then (c) becomes a special case of (e).

(ii) \emph{The ALVIN variant
(the reversed J{\'o}nsson condition).} Now give another look from scratch at
Condition (ii) in \ref{gummt}. 
There is no big difference
between conditions (b)-(c)
for $p$ and conditions
(d)-(f) for the $j_i$'s.  
However, in addition, the $j_i$'s
are supposed to satisfy (a).
What if we require
$p$ to satisfy the analogue of  (a)?
That is, what  if we ask that $p$ satisfies  $x=p(x,y,x)$, too?
At a very first glance, the answer would be that we get the condition for 
$n+2$-distributivity. Not precisely!
Formally, we have to add another
term at the beginning, to be interpreted as
the trivial projection onto the first coordinate.
This term should be numbered as $j _{-1} $,
so we have to shift the indices by  $1$   
and what we get is 
like $n+2$-distributivity
but with the role of
even and odd exchanged.
This is sometimes called the ALVIN variant of 
J{\'o}nsson condition, since it appears in \cite{MMT}. See also \cite{FV}.
We shall  call this condition
$n+2$-distributivity$^\smallsmile $,
or say that $\mathcal V$ has $n+3$ J{\'o}nsson$ ^\smallsmile $   
terms.

(iii) \emph{Comparing the J{\'o}nsson 
and the ALVIN conditions.}
 Notice that changing the role of odd and even in J{\'o}nsson condition
 might appear an innocuous variant, but  
this is far from being true, if we keep $n$ fixed. Indeed,
ALVIN 2-distributivity means having a Pixley term,
that is, arithmeticity!
Let us remark that, however, if $n$ 
is odd, the J{\'o}nsson and ALVIN 
$n$-distributivity 
conditions are equivalent, just
reverse both the order of terms and of variables. 
In a sense, the arguments in the present remark suggest
that  J{\'o}nsson condition for $n$ odd, as well as
 the ALVIN condition for every $n$,  share some ``spurious''
aspects of permutability, as already evident in the case of Pixley
(ALVIN $2$) terms. We shall make an intense use 
of this permutability aspect at various places.
Cf.\ also the proof of \cite[Theorem 2.3]{jds}.

 For the reader who prefers to see the above Maltsev conditions
expressed in terms of congruence identities,
$n$-distributivity corresponds to the identity
$ \alpha ( \beta \circ \gamma ) \subseteq 
\alpha \beta   \circ_{n}
 \alpha \gamma  $ while ALVIN
$n$-distributivity corresponds to
$ \alpha ( \beta \circ \gamma ) \subseteq 
\alpha \gamma    \circ_{n}
 \alpha \beta   $.
As we mentioned, having $n$ Gumm terms corresponds to 
$ \alpha ( \beta \circ \gamma ) \subseteq 
\alpha ( \gamma \circ \beta ) \circ (\alpha \gamma  \circ_{n-2}
 \alpha \beta  ) $.

Notice that if $n$ is even,   J{\'o}nsson condition is symmetrical, namely,
we get the same condition if we exchange both the order of terms and of variables.
On the other hand,  J{\'o}nsson condition is not symmetrical
when $n$ is odd, in the sense that the first and the last term
play different roles.  Compare \cite[Remark 2.2]{ia} for the case
$n=3$; similar arguments apply to the case $n$ odd, $n>3$.

(iv) \emph{Gumm terms as defective ALVIN 
distributivity terms.} Of course, however, Gumm's condition,
in the form presented here, 
does \emph{not} require
the equation $x=p(x,y,x)$. Hence, in the above sense,
Gumm terms can be interpreted as
 \emph{defective} ALVIN  terms,
since one equation of the form 
$x=j_i(x,y,x)$, is missing.
We shall see  that certain arguments 
which work for congruence distributive varieties
can be still applied when one equation of that form is missing.
However, it is necessary that the missing equation  lies
at one end of the chain of terms, i.e., it involves 
the first (or perhaps sometimes the last) nontrivial term.
In some cases we might even allow \emph{two}
missing equations, see Proposition \ref{rmk1}. 
 
As a particularly appealing and sharp application 
of the above ideas, we have proved 
in \cite[Theorem 2.3]{jds}
that a variety with $3$ Gumm terms
satisfies 
$ \alpha( \beta \circ _{m+2}  \gamma  )   =
  \alpha ( \beta \circ \gamma )    \circ (\alpha \beta  \circ_m \alpha \gamma)$,  for every $m \geq 2$.   
Moreover, in \cite[Theorem 3.5]{ia}
we have showed that a variety  with $3$ Gumm terms
satisfies $R(S \circ RT) \subseteq 
RS \circ RT \circ RT \circ RS$, actually, even stronger identities.
Notice that terms are counted in a different way in \cite{ia},
so that  what we call here a variety with $3$ Gumm terms
is called in \cite{ia} a variety with $2$ Gumm terms.  
\end{remark}

The following results extend \cite[Lemma 4.1]{jds} 
and use arguments similar to those used in \cite[Section 2]{ricm}.
We describe our man aim. 
Tschantz \cite{T} showed that a congruence modular variety 
satisfies 
$ \alpha ( \beta +  \gamma ) =
 \alpha ( \beta \circ \gamma ) \circ ( \alpha \beta + \alpha \gamma )$.
It follows from the general theory
of Maltsev conditions that, for every congruence modular variety $\mathcal V$ 
and every  $m$, there is some $k$ such that 
$ \alpha ( \beta \circ _m  \gamma ) \subseteq
 \alpha ( \beta \circ \gamma ) \circ ( \alpha \beta \circ _k \alpha \gamma )$
holds in $\mathcal V$;
actually, 
$k$ can be expressed in function 
only 
of $m$ and on  the number
of Gumm terms, but otherwise $k$ does not depend on $\mathcal V$.
This follows also from an accurate analysis of
Tschantz proof.  
Here we provide a better evaluation than the one given by Tschantz arguments,
but we absolutely do not know whether our evaluation gets close to be optimal. 
If $\mathbf A$ is an algebra and $X \subseteq A $,
we let $ \overline{X} $ denote the smallest 
 admissible relation  on $\mathbf A$ 
containing $X$.  

\begin{lemma} \labbel{appgumm}
Suppose that the variety $\mathcal V$ has $n+2$ Gumm terms
$p, j_1, \dots, j_{n+1}$.

  \begin{enumerate}[(i)]    
\item  
The variety  $\mathcal V$ satisfies
\begin{equation}\labbel{aga}
\alpha (R \circ S) \subseteq 
\alpha ( \overline{R^ \smallsmile  \cup S  })
\circ 
\big(
( \alpha R \circ \alpha S) \circ_{n}
 ( \alpha S^ \smallsmile  \circ \alpha R ^ \smallsmile )
\big) , 
  \end{equation}    
for every congruence $\alpha$ and admissible relations  
$R$ and $S$.
\item
More generally,
if $\alpha$ is a congruence, 
$T_1$, \dots, $T_m$ are  admissible relations  
and $R$ and $S$   are binary relations such that 
$R \circ S \supseteq T_1 \circ \dots \circ T_m$,  
then
\begin{multline}\labbel{ag}
\alpha ( T_1 \circ  T_2 \circ \dots \circ T_{m} ) \subseteq 
\\ 
\alpha ( \overline{R^ \smallsmile  \cup S  })
\circ
\big( ( \alpha T_1 \circ \alpha  T_2 \circ \dots \circ \alpha  T_{m} )
\circ _{n}
( \alpha T_m^ \smallsmile  \circ \dots \circ \alpha  T_2 ^ \smallsmile
\circ \alpha  T_{1} ^ \smallsmile)
\big)
 \end{multline} 
 \end{enumerate} 
\end{lemma}

 \begin{proof}
We first prove \eqref{aga}.
Suppose that
$(a,c) \in \alpha (R \circ S)$,
thus
$ a\mathrel \alpha  c$
and there is $b$ such that 
$ a \mathrel {R } b \mathrel {S } c $.
Then 
$a= p(a,b,b) \mathrel {\overline{R^ \smallsmile  \cup S  } }
p(a,a,c)  $
and 
$ a= p(a,a,a) \mathrel \alpha p(a,a,c) $,
thus
$(a, j_1(a,a,c))  = (a, p(a,a,c))  \in \alpha ( \overline{R^ \smallsmile  \cup S  }) $.

The rest is quite standard.
We have $j_i(a,a,c) \mathrel R j_i(a,b,c) \mathrel S j_i(a,c,c) $
and also  
$j_i(a,a,c) \mathrel \alpha j_i(a,a,a) = a = j_i(a,b,a) 
\mathrel \alpha  j_i(a,b,c)$,
thus
$j_i(a,a,c) \mathrel {\alpha R}  j_i(a,b,c)$
and similarly
$j_i(a,b,c) \mathrel {\alpha S}  j_i(a,c,c)$,
thus, for $i$ odd, 
$j_i(a,a,c) \allowbreak 
\mathrel {\alpha R \circ  \alpha S}  j_{i}(a,c,c) =j_{i+1}(a,c,c)$.
For $i$ even, we compute
 $j_i(a,c,c) \mathrel {S^ \smallsmile}  j_i(a,b,c)
 \mathrel {R^ \smallsmile } j_i(a,a,c) $
and, arguing as above,
 $j_i(a,c,c) \mathrel { \alpha S^ \smallsmile \circ 
\alpha R^ \smallsmile } j_i(a,a,c)=j_{i+1}(a,a,c) $.

In conclusion, the elements
$p(a,a,c)=j_1(a,a,c)$, 
$j_{1}(a,c,c) =j_{2}(a,c, c)$,
$j_{2}(a,a, \allowbreak c) =j_{3}(a,a,c)$,
\dots 
witness
$(a,c) \in \alpha ( \overline{R^ \smallsmile  \cup S  })
\circ 
\big(
( \alpha R \circ \alpha S) \circ_{n} ( \alpha S^ \smallsmile  \circ \alpha R ^ \smallsmile )
\big)$
and the identity \eqref{aga} is proved. 

The proof of \eqref{ag} essentially contains no new idea. 
If 
$(a, c) \in \alpha ( T_1 \circ  T_2 \circ \dots \circ T_{m} ) $,
then $a \mathrel \alpha  c$ and there are elements
$a=b_0$, $b_1$, \dots, $b_m = c$
such that 
$b_0 \mathrel { T_1 } b_1 \mathrel {   T_2 } \dots \mathrel {  T_{m}} b_m $.
Moreover, since 
$ T_1 \circ \dots \circ T_m \subseteq R \circ S$,  
then
$ a \mathrel {R } b \mathrel {S } c $, for some $b$.

The first part of the above proof shows that
$ (a, p(a,a,c))  \in \alpha ( \overline{R^ \smallsmile  \cup S  }) $.
Moreover, for every odd $i$ and for $0 \leq  \ell  < m$, we have
$j_i (a,b_ \ell , c) \mathrel { T_{ \ell +1}} j_i (a,b_{ \ell +1}, c) $.
In addition,
$j_i (a,b_ \ell , c) \mathrel \alpha j_i (a,b_ \ell , a)=a=  j_i (a,b_{ \ell +1}, a)
\mathrel \alpha    j_i (a,b_{ \ell +1}, c) $,
thus 
$j_i (a,b_ \ell , \allowbreak  c) \mathrel { \alpha  T_{ \ell +1}} j_i (a,b_{ \ell +1}, c) $.
This shows that, for $i$ odd,
$(j_i(a,a,c), j_i(a,c,c) )
\in  
\alpha T_1 \circ \alpha  T_2 \circ \dots \circ \alpha  T_{m} $.     
Similarly, for $i$ even
$(j_i(a,c,c), \allowbreak  j_i(a,a,c) ) \in 
 \alpha T_m^ \smallsmile  \circ \dots \circ \alpha  T_2 ^ \smallsmile
\circ \alpha  T_{1} ^ \smallsmile$.
 
Hence the same elements as above, that is, 
$j_1(a,a,c)$, 
$j_{2}(a,c,c)$,
$j_{3}(a,a,c)$,
\dots \ 
witness
that 
$(a,c)$ belongs to the right-hand side of
\eqref{ag}. 
 \end{proof}    

When $n$ is even we can apply the trick at the beginning of
the proof of \ref{appgumm} ``at both ends''.
Remarkably enough, there is no need of the equation
$x=j_{n}(x,y,x) $ in order to carry the argument over.
Let us say that a variety $\mathcal V$ has 
\emph{$n+2$ defective Gumm terms} 
if $\mathcal V$ has terms satisfying Condition (ii)
 in Theorem \ref{gummt}, 
except that equation 
$x=j_{n}(x,y,x) $ is not assumed.

\begin{proposition}  \labbel{rmk1}
If $\mathcal V$ has $n+2$ defective Gumm terms and $n$ is even, 
then $\mathcal V$ satisfies the following identity.
\begin{equation*}\labbel{agaX}
\alpha (R \circ S) \subseteq 
\alpha ( \overline{R^ \smallsmile  \cup S  })
\circ 
\big(
( \alpha R \circ \alpha S) \circ_{n-1}
 ( \alpha S^ \smallsmile  \circ \alpha R ^ \smallsmile )
\big) \circ  \alpha (\overline{R \cup S ^\smallsmile   }) 
  \end{equation*}

Moreover, under the assumptions  in 
 Lemma \ref{appgumm}(ii) about
$\alpha, T_1, \dots, T_m, R$ and $S$, we have that
$\mathcal V$ satisfies the following identity.
\begin{multline*}
\alpha ( T_1 \circ  T_2 \circ \dots \circ T_{m} ) \subseteq 
\\
\alpha ( \overline{R^ \smallsmile  \cup S  })
\circ
\big( ( \alpha T_1 \circ \alpha  T_2 \circ \dots \circ \alpha  T_{m} )
\circ _{n-1}
( \alpha T_m^ \smallsmile  \circ \dots \circ \alpha  T_2 ^ \smallsmile
\circ \alpha  T_{1} ^ \smallsmile)
\big) \circ  \alpha (\overline{R \cup S ^\smallsmile   }) 
 \end{multline*} 
\end{proposition}

\begin{proof} 
If $a \mathrel R b \mathrel S c$, then
$j_{n}(a,c,c) \mathrel { \overline{R \cup S ^\smallsmile   } } j_n(b,b,c)=
j_{n+1}(b,b,c) =c$, since $n$ is even.
Moreover, if 
$a \mathrel \alpha  c$, then
$j_n(a,c,c) \mathrel \alpha  j_n(c,c,c) =c$.  
All the rest is like \ref{appgumm}.
\end{proof}

\begin{theorem} \labbel{agt} 
Suppose that  $\mathcal V$ has $n+2$ Gumm terms 
$p, j_1, \dots, j_{n+1}$
and 
$m$, $h$  are natural numbers.
 
 If 
 $m=4h+2$,
then the following identities hold throughout $\mathcal V$.  
\begin{align}\labbel{4hb} 
\alpha ( \beta \circ  \gamma \circ {\stackrel{m+1}{\dots}}
\circ \beta )
& \subseteq 
\alpha (\gamma \circ
 \beta  \circ {\stackrel{2h+2}{\dots}} \circ \beta )
\circ
(\alpha \gamma  \circ_{mn} \alpha \beta )
\\
\labbel{4hbconv} 
\alpha ( \beta \circ  \gamma \circ {\stackrel{m+1}{\dots}}
\circ \beta ) 
& \subseteq 
(\alpha \beta   \circ_{mn} \alpha \gamma  )
\circ
\alpha (\beta  \circ
 \gamma   \circ {\stackrel{2h+2}{\dots}} \circ \gamma  )
  \end{align} 

 If 
 $m=4h$, 
then the following identity holds throughout $\mathcal V$.  
\begin{equation}\labbel{4h} 
\alpha ( \beta \circ  \gamma \circ {\stackrel{m+1}{\dots}}
\circ \beta )
 \subseteq 
 \alpha (\beta \circ   \gamma  \circ {\stackrel{2h+1}{\dots}} \circ \beta )
\circ
(\alpha \gamma  \circ_{mn} \alpha \beta )  
  \end{equation} 
\end{theorem} 

\begin{proof}
In order to prove \eqref{4hb} and \eqref{4h}, 
apply the identity \eqref{ag} in Lemma \ref{appgumm}
with $m+1$ in place of $m$,  $T_1= T _3 = \dots = T_{m+1} =\beta $  
and
$T_2= T _4 = \dots =\gamma  $
(we shall soon define $R$ and $S$).
Formula  \eqref{ag} in Lemma \ref{appgumm}
gives $(m+1)n$ factors of the form
$\alpha T_{\ell} $, but we have 
$n-1$ contiguous pairs of the form
$ \alpha \beta \circ \alpha \beta $,
hence, after a simplification, the number of factors of the form
$\alpha T_{\ell} $ reduces to 
$(m+1)n -(n-1) = mn+1$.

Then 
in the case  $m=4h+2$,
take $R= \beta \circ  \gamma  \circ {\stackrel{2h+1}{\dots}} 
\circ \beta $ 
and
 $S=   \gamma \circ \beta \circ {\stackrel{2h+2}{\dots}} \circ \beta $.
In the case 
  $m=4h$
 take 
$R= \beta \circ   \gamma  \circ {\stackrel{2h}{\dots}} 
\circ  \gamma   $ 
and
 $S=   \beta  \circ  \gamma \circ {\stackrel{2h+1}{\dots}} \circ \beta $.
Thus in both cases we have
$R \circ S =  \beta \circ  \gamma \circ {\stackrel{m+1}{\dots}}
\circ \beta 
\supseteq 
T_1 \circ \dots \circ T _{m+1} $ and   
$\overline{R ^ \smallsmile  \cup S } = S $. 

Then, say, 
$\alpha (\gamma \circ
 \beta  \circ {\stackrel{2h+2}{\dots}} \circ \beta )
\circ
\alpha \beta$
is obviously equal to 
$\alpha (\gamma \circ
 \beta  \circ {\stackrel{2h+2}{\dots}} \circ \beta )$,
hence in both cases one more factor is absorbed and we get the conclusions.

Finally, \eqref{4hbconv} is obtained by taking the converse of both sides 
of \eqref{4hb}, since $m+1$ is odd and both $2h+2$ and $mn$ are even.  
\end{proof}

\begin{corollary} \labbel{agtcor2}
 Suppose that $\mathcal V$ has $n+2$
Gumm terms and $\mathcal V$ satisfies the identity
\begin{equation}\labbel{2}  
     \alpha (\beta \circ   \gamma  \circ {\stackrel{2r+1}{\dots}} \circ \beta )
\subseteq
\alpha ( \gamma \circ \beta ) 
\circ (\alpha \gamma  \circ_s \alpha \beta  ),
 \end{equation}
for some natural numbers $r, s \geq 1$. 
Then  
$\mathcal V$ also satisfies 
\begin{equation} \labbel{2a}  
\alpha (\beta \circ   \gamma  \circ {\stackrel{4r+1}{\dots}} \circ \beta )
\subseteq
\alpha ( \gamma \circ \beta ) 
\circ (\alpha \gamma  \circ_{s+4rn} \alpha \beta  )
  \end{equation}     
and, more generally, for every $q \geq 1$,  $\mathcal {V}$ satisfies
\begin{equation} \labbel{2b}  
\alpha (\beta \circ   \gamma  \circ {\stackrel{2^q r+1}{\dots}} \circ \beta )
\subseteq
\alpha ( \gamma \circ \beta ) 
\circ (\alpha \gamma  \circ_{s+(2 ^{q+1}-4) rn } \alpha \beta  ).
  \end{equation}
 \end{corollary}   

\begin{proof}
Identity \eqref{2a}  is immediate from 
identity \eqref{4h} in  Theorem \ref{agt}, taking
$h=r$ and then using \eqref{2}.

The identity \eqref{2b} follows by induction.
The case $q=1$ is the assumption \eqref{2}. 
Suppose that 
\eqref{2b} holds for some $q \geq 1$.
Taking  $r' = 2 ^{q-1}r  $ and
$s'= s+(2 ^{q+1}-4)rn$
we have that \eqref{2} holds with
$r'$, $s'$ in place of, respectively,  $r$, $s$.  
By applying \eqref{2a}, we get
an inclusion with parameters 
$4r'+1= 2 ^{q+1}r +1$ and $s'+4r'n = s+(2 ^{q+1}-4) rn + 2 ^{q+1}rn
= s+(2 ^{q+2}-4 ) rn $, what we needed.   
\end{proof}

\begin{theorem} \labbel{agtcor}
If $\mathcal V$ has $n+2$ Gumm terms $p, j_1, \dots, j_{n+1}$, then  
\begin{align}\labbel{qdist} 
\alpha ( \beta \circ  \gamma \circ {\stackrel{2^q+1}{\dots}}
\circ \beta )
& \subseteq 
\alpha (\gamma \circ
 \beta )
\circ
(\alpha \gamma  \circ_{k} \alpha \beta ) \quad  \text{ and } 
\\
\labbel{qdistconv} 
\alpha ( \beta \circ  \gamma \circ {\stackrel{2^q+1}{\dots}}
\circ \beta )
& \subseteq 
(\alpha \beta   \circ_{k} \alpha \gamma  ) \circ
\alpha ( \beta  \circ \gamma  )
 \end{align} 
hold in $\mathcal V$, for every natural number $q\geq1$
and where $k=(2 ^{q+1}-2)n$. 
 \end{theorem}

 \begin{proof}
By the identity \eqref{4hb}
in Theorem \ref{agt} with $h=0$
we get that  
the identity \eqref{2} in \ref{agtcor2}
holds with $r=1$ and $s= 2n$.
Thus \eqref{qdist} follows from \eqref{2b}.
By taking the converse of both sides in \eqref{qdist},
we get \eqref{qdistconv}, since $2^q+1$ is odd and  $k$ is even.  
\end{proof}  

If we replace $\gamma$ by $\alpha \gamma $ in
 identity \eqref{qdistconv} in
Theorem \ref{agtcor},
we get results about
$(m, k)$-modularity,
since $\alpha( \beta \circ \alpha \gamma  ) = 
\alpha \beta  \circ \alpha \gamma  $.
However, the trick can be performed
at each induction step, hence we obtain a different 
kind of results.
Formally, the two methods seem to provide incomparable evaluations,
but the latter method provides ostensibly better bounds.
For example,  by taking $q=4$
in (i) and (ii) below, 
from $n+2$ Gumm terms  we get,
respectively,
$(17, 30n+2)$-modularity and  $(15, 22n+2) $-modularity, 
and the latter appears to be a nicer evaluation, though formally incomparable. 

\begin{theorem} \labbel{qkmod}
Suppose that $\mathcal V$ has $n+2$ Gumm terms $p, j_1, \dots, j_{n+1}$.
  \begin{enumerate}[(i)]    
\item   
 For every $q \geq 1$, 
$\mathcal V$ is
$(2^q +1 , (2 ^{q+1}-2) n + 2)$-modular.
\item
 For every  $q \geq 2$,  
$\mathcal V$ is 
$(2^q -1, (2^{q+1}-2q-2)n+2)  $-modular.
 \end{enumerate}
 \end{theorem} 

\begin{proof} 
As we mentioned, 
(i) is immediate from 
identity \eqref{qdistconv} in
Theorem \ref{agtcor} by
taking $\alpha \gamma $ in
place of $\gamma$.

Item (ii) is proved by induction on $q$.
If $q=2$, then the identity \eqref{4hbconv} in
Theorem \ref{agt} with $h=0$ and $\alpha \gamma $ in place of $\gamma$  
gives   $(3,2n+2)  $-modularity, 
since $\alpha( \beta \circ \alpha \gamma  ) = \alpha \gamma \circ \alpha \beta $. 
The base step is thus proved. By the way, 
in this special case, the method is essentially known,
compare \cite[p.\ 172]{D} and the proofs of \cite[Theorem 7.4(iv) $\Rightarrow $  (i)]{G}, \cite[Theorem 1 (3) $\Rightarrow $  (1)]{LTT}.
The main point in the present theorem   is the evaluation of some bounds for larger values of $q$.

Suppose that the theorem  holds for some
$q \geq 2$, thus we have 
$ \alpha ( \beta \circ \alpha \gamma \circ {\stackrel{2^q-1}{\dots}} \circ \beta )
 \subseteq 
\alpha \beta   \circ_k \alpha \gamma $,
with $k=(2^{q+1}-2q-2)n+2$. 
Apply the identity \eqref{4hbconv} in Theorem \ref{agt}
with $h=2^{q-1}-1$, thus $2h+2=2^q$ and  $m=2 ^{q+1}-2 $. We get    
\begin{multline*}
\alpha ( \beta \circ \alpha \gamma \circ {\stackrel{2^{q+1}-1}{\dots\dots}} \circ \beta )
 \subseteq ^{\eqref {4hbconv}} 
( \alpha \beta  \circ _{mn} \alpha \gamma   )
\circ
\alpha ( \beta \circ \alpha \gamma  \circ {\stackrel{2^{q}}{\dots}}
\circ \beta  \circ \alpha \gamma  ) 
=
 \\
( \alpha \beta  \circ _{mn} \alpha \gamma   )
\circ 
\alpha (\beta \circ \alpha \gamma \circ {\stackrel{2^{q}-1}{\dots}} \circ \beta ) \circ \alpha \gamma 
 \subseteq ^{\text{ih}} 
\\
( \alpha \beta  \circ _{mn} \alpha \gamma   )
\circ 
( \alpha \beta  \circ _{k} \alpha \gamma   )
\circ 
\alpha \gamma 
= 
\alpha \beta   \circ _{k + mn} \alpha \gamma ,
\end{multline*}   
since both $mn$ and  $k$ are even and,
as  already used several times,
$\alpha \gamma \circ \alpha \gamma = \alpha \gamma $.  
Recalling that 
$k=(2^{q+1}-2q-2)n+2$ and 
$m=2 ^{q+1}-2 $, we get 
$k + mn = (2 ^{q+2} -2(q+1)-2 ) n+2$, 
thus the induction step is complete.
\end{proof}  

Of course, as in Corollary \ref{agtcor2}, we can provide a version 
of Theorem \ref{qkmod} starting with some fixed identity. 

\begin{corollary} \labbel{qkmod2}
Suppose that $\mathcal V$ has $n+2$ Gumm terms,
$h \geq 1$, $t$ is even  
 and
\begin{equation*}
\alpha ( \beta \circ  \alpha \gamma \circ {\stackrel{2h+1}{\dots}}
\circ \beta )
 \subseteq 
\alpha \beta   \circ_{t} \alpha \gamma  
  \end{equation*} 
holds in $\mathcal V$. Then $\mathcal V$ also satisfies
\begin{equation}\labbel{qmod2} 
\alpha ( \beta \circ  \alpha \gamma \circ {\stackrel{4h+3}{\dots}}
\circ \beta )
 \subseteq 
\alpha \beta   \circ_{t+(4h+2)n} \alpha \gamma  
  \end{equation} 
and, more generally, for every $p \geq 1$,
$\mathcal V$ satisfies 
\begin{equation}\labbel{qmod3} 
\alpha ( \beta \circ  \alpha \gamma \circ {\stackrel{z}{\dots}}
\circ \beta )
 \subseteq 
\alpha \beta   \circ_{t'} \alpha \gamma 
 \end{equation}
 with
 $z=  2 ^{p}(h+1) -1 $ and $ t'=  t +(2 ^{p+1}-4) hn  + (2 ^{p+1}- 2p -2 )n$. 
\end{corollary}

\begin{proof}
Identity \eqref{qmod2} is immediate from identity \eqref{4hbconv}
 in Theorem \ref{agt}, using
transitivity of $\alpha \gamma $ still another time. 
Identity \eqref{qmod3} follows from an induction 
(on $p$)
similar to the proof of
Theorem \ref{qkmod}.
\end{proof}

We now observe that we can improve Lemma \ref{appgumm} a bit.
 
\begin{proposition} \labbel{agaimpr}
If $\mathcal V$ has $n+2$ Gumm terms, then,
under the assumptions in Lemma \ref{appgumm},
$\mathcal {V}$ satisfies 
\begin{gather}\labbel{agai}
\alpha T \circ \alpha (R \circ S) \subseteq 
\alpha ( \overline{T \cup R^ \smallsmile  \cup S  })
\circ 
\big(
( \alpha R \circ \alpha S) \circ_{n}
 ( \alpha S^ \smallsmile  \circ \alpha R ^ \smallsmile )
\big), \text{ and } 
\\ 
\begin{aligned}     
\labbel{agi}
 \alpha T \circ  \alpha ( T_1 \circ  T_2 \circ \dots \circ T_{m} ) 
 & \subseteq  
\\
 \alpha ( \overline{T \cup R^ \smallsmile  \cup S  })
\circ
\big( ( \alpha T_1 \circ \alpha  T_2 &
 \circ \dots \circ \alpha  T_{m} )
\circ _{n}
( \alpha T_m^ \smallsmile  \circ \dots \circ \alpha  T_2 ^ \smallsmile
\circ \alpha  T_{1} ^ \smallsmile)
\big).
\end{aligned} 
\end{gather}
 \end{proposition}  

\begin{proof}  
If 
$ a' \mathrel T  a \mathrel {R } b \mathrel {S } c $, then
$a'= p(a',b,b) \mathrel {\overline{T \cup R^ \smallsmile  \cup S  } }
p(a,a,c)  $; all the rest goes as in 
Lemma
\ref{appgumm}.
The proof of 
\eqref{agi} 
is similar to \eqref{ag}.
\end{proof}

Proposition  \ref{rmk1}, too, can be improved in a corresponding  way.
In the case of Proposition  \ref{rmk1} we can even add some factor $\alpha T'$
on the right in the left-hand side. We leave details to the
interested reader.   

\begin{remark} \labbel{agaimprem}
By \eqref{agai}, the identity \eqref{4hb} in 
Corollary
 \ref{agt}
can be improved to   
\begin{equation*}\labbel{4hbi}
\alpha ( \beta \circ  \gamma \circ {\stackrel{2h+1}{\dots}}
\circ \beta ) \circ 
\alpha ( \beta \circ  \gamma \circ {\stackrel{m+1}{\dots}}
\circ \beta )
 \subseteq 
\alpha (\gamma \circ
 \beta  \circ {\stackrel{2h+2}{\dots}} \circ \beta )
\circ
(\alpha \gamma  \circ_{mn} \alpha \beta ),
  \end{equation*}     
under the same assumptions and using the same proof as in Corollary
 \ref{agt}.

 On one hand, this seems to be a significant improvement;
 on the other hand,  it seems to provide only slightly
better bounds  for  the kind of  identities we are considering here.

As a way of example, we shall present an application of 
  \eqref{agi} in a specific numerical case.
If $\mathcal V$ has $n+2$ Gumm terms, then
 the case $h=1$, $m=4$ in  identity \eqref{4h} in Theorem \ref{agt},
taking converses,
gives $ \alpha ( \beta \circ \gamma \circ \beta  \circ \gamma \circ \beta ) \subseteq 
( \alpha \beta  \circ _{4n} \alpha \gamma  ) 
\circ \alpha ( \beta \circ \gamma \circ \beta )=
( \alpha \beta  \circ  \alpha \gamma  \circ  {\stackrel{4n-1}{\dots}} \circ \alpha \beta )
 \circ \alpha \gamma  
\circ \alpha ( \beta \circ \gamma \circ \beta )
$. 
Then, using \eqref{agi} with $m=3$, 
$T= \alpha \gamma  $, 
$R= \beta $, 
$S= \gamma \circ \beta $,  
$T_1=T_3= \beta $ and  $T_2= \gamma $, we get   
$ \alpha \gamma   
\circ \alpha ( \beta \circ \gamma \circ \beta )
\subseteq 
\alpha ( \gamma \circ \beta ) \circ ( \alpha \gamma \circ_{2n} \alpha \beta )$.
Summing up, we get
\begin{equation}\labbel{bbb}     
 \alpha ( \beta \circ \gamma \circ \beta  \circ \gamma \circ \beta ) \subseteq 
( \alpha \beta \circ _{4n-1} \alpha \gamma   ) \circ 
\alpha ( \gamma \circ \beta ) \circ ( \alpha \gamma \circ_{2n} \alpha \beta ).
  \end{equation}
If we now take $\alpha \gamma $ in place of $\gamma$
in \eqref{bbb},  we get
 $ \alpha ( \beta \circ \alpha  \gamma \circ \beta  \circ \alpha  \gamma \circ \beta )
 \subseteq 
 \alpha \beta  \circ_{6n+1} \alpha \gamma  $.
This is a slight improvement on the case $q=2$ in Theorem \ref{qkmod}(i),
which shows that  
 $ \alpha ( \beta \circ \alpha  \gamma \circ \beta  \circ \alpha  \gamma \circ \beta )
 \subseteq 
 \alpha \beta   \circ_{6n+2} \alpha \gamma  $.

Going some steps further, we 
shall have a ``longer'' $R^ \smallsmile  \cup S $,
hence we will be able to annihilate more  factors
of the form $\alpha \beta $   or $\alpha \gamma $ at a time, since, say,
$ \alpha \beta \circ \alpha \gamma \subseteq \beta \circ \gamma \circ \beta $.
However, the improvements we get this way appear so small
that we have not worked out explicit details. 
 \end{remark}

\section{Joining the two approaches and some mystery} \labbel{probs}

Recall that 
$D_{ \mathcal V} (m) $  
is the smallest 
 $h$ such that $\mathcal V$ is  $(m,h)$-modular,
that is, 
 the smallest 
 $h$ such that $\mathcal V$
satisfies the congruence identity 
$ \alpha ( \beta \circ_m \alpha \gamma) \subseteq 
\alpha \beta \circ_h \alpha \gamma$.
We have found bounds for
$D_{ \mathcal V} (m) $  
using two relatively different methods
in Theorems \ref{thm} and  
 \ref{qkmod}. 
The problem naturally arises to see
whether one method is better than the other.
So far, this appears a quite mysterious question.

First,  notice that there are trivial cases,
for example,  permutable varieties (here, of course, ``trivial'' is intended only relative to the problems we are discussing;
permutable varieties  by themselves are far from being trivial!).
As already noticed 
by Day \cite[Theorem 2]{D},
 congruence permutability corresponds exactly to the case in which $\mathcal V$ has $3$ Day terms.
In this case, 
 $(m,2)$-modularity holds, for every $m$.
This case corresponds exactly also to $\mathcal V$ having $2$ Gumm terms.
More generally,
in $n$-permutable modular  varieties, 
 $(m,n)$-modularity holds, for every $m$. 

In general
and, for simplicity, with some approximation,
 Theorem \ref{thm} gives a bound
roughly of the form $ h = 2r ^{q-1} $ for
$(2 ^{q} , h)$-modularity, for varieties 
with $2r+1$ Day terms. 
On the other hand,
Theorem \ref{qkmod}
gives a bound
roughly of the form $ h = n2 ^{q+1}  $ for
$(2 ^{q} , h)$-modularity, for varieties 
with $n+2$ Gumm terms. 
In the following discussion, suppose  that $r>2$
(for $3$-modular varieties  
 the best possible results are given by
Proposition  \ref{small}
and perhaps this is the case also for  $4$-modular
varieties).
Then 
$  2r ^{q-1} $
 grows asymptotically faster than
$  n2 ^{q+1}  $, as $q$ increases, no matter 
the  value of $n$.  
Hence, fixed some variety $\mathcal V$ and for a sufficiently large $q$, 
we have that Theorem \ref{qkmod}
provides a better bound for $D_{ \mathcal V} (2^q) $. 
By a trivial monotonicity property,
this holds also for sufficiently large $m$,
when evaluating $D_{ \mathcal V} (m) $.
In other words, for large $m$, Theorem \ref{qkmod} \emph{alone}
provides a better bound for $D_{ \mathcal V} (m) $,
in comparison with    
Theorem \ref{thm} \emph{alone}.

However, 
it might happen that, for some relatively small value of $m$, 
Theorem \ref{thm} gives a better bound for $D_{ \mathcal V} (m) $,  
in comparison with  Theorem \ref{qkmod}.
As far as $m$ increases, the evaluation of $D_{ \mathcal V} (m) $ 
obtained 
according to  Theorem \ref{thm} becomes worse and worse,
in comparison with the evaluation given by Theorem \ref{qkmod}.
However, at a certain point, rather than applying 
Theorem \ref{qkmod} ``from the beginning'', we can instead use 
the related Corollary \ref{qkmod2} 
when this becomes more convenient, 
by using the evaluation given by 
 Theorem \ref{thm} as the premise in \ref{qkmod2}. 
In general, the above procedure seems to provide better bounds for $D_{ \mathcal V} (m) $  
than Theorem \ref{qkmod}  alone. 

The above arguments
furnish the following corollary,
which holds for every even $k=2r$, but probably has some advantage only 
for $k > 4$.
 Notice that, curiously enough, though we are starting 
from the beginning with the rather 
complicated and dissimilar formulae given by Theorem \ref{thm} 
 and Corollary  \ref{qkmod2},
at the end some expressions simplify in such a way that we obtain  a somewhat
readable expression!

\begin{theorem} \labbel{comb}
Suppose that  $ r, p, q \geq 1$,  $\mathcal V$ 
has  $2r+1$ Day terms and $n+2$ Gumm terms.  Then
$\mathcal V$ is
 $( z,   w)$-modular, with $z=  2 ^{p}2 ^{q}  -1 $ and 
$ w =   2r^q + (2 ^{q}2 ^{p+1} -2 ^{q+2}-2p+2)n $. 
 \end{theorem}

 \begin{proof}
Because of Theorem \ref{thm}, the assumption in 
Corollary \ref{qkmod2} can be applied with
$h= 2 ^{q} -1  $ and $t =  2r^q  $.
Then identity \eqref{qmod3} in \ref{qkmod2} gives 
 $z=  2 ^{p}2 ^{q}  -1 $ and 
$ w =    2r^q +
( 2 ^{q} -1 )(2 ^{p+1}-4) n + (2 ^{p+1}- 2p -2 )n=
  2r^q + (2 ^{q}2 ^{p+1} -2 ^{q+2}-2p+2)n $.
 \end{proof}  

Small improvements on Theorem \ref{qkmod} and on 
the identities \eqref{2b},  \eqref{qdist}, \eqref{qdistconv} and \eqref{qmod3} 
in Theorem \ref{agtcor}, Corollaries  \ref{agtcor2} and \ref{qkmod2}
can be obtained by employing the trick 
used in the proof of identity (1) in 
\cite[Theorem 2.1]{ricm}.  Roughly,
we have to write explicitly the nested terms arising from the
proofs  and ``move together'' many variables at a time. 
This seems to provide only small improvements,
while, on the other hand, we seem to get more and more involved 
proofs and formulae, as the nesting level increases.
Perhaps, similar arguments apply also to 
Theorem \ref{thm} and Proposition \ref{thm2}.
As we mentioned, further small improvements
can be obtained using Proposition  \ref{agaimpr}.
A further minimal improvement can be obtained 
when $n$ is even by using Proposition \ref{numdd} below.
By the above arguments,  Theorem  \ref{comb}
can be slightly improved, too.  

Apart form the possible above small improvements, we do not know whether the situation
depicted in Theorem  \ref{comb} is the best possible one.
This might be the case, but, on the 
other hand, it might happen that either  Theorem \ref{thm} or
Theorem  \ref{qkmod} 
can be improved in such a way that one of them always gives the best 
evaluation.
Or perhaps there is a completely new method 
which is always the best; or even there is no ``best''  method
working in every situation and 
 everything depends on some particular property
of the specific variety at hand.

Anyway, what lies 
at the heart of our arguments are  the numbers
of Day terms in Section \ref{b}  and of Gumm terms in Section \ref{gumm}. 
These numbers are tied by some constraints,
as theoretically implicit from the theory of Maltsev conditions.
More explicitly, Lakser,  Taylor and  Tschantz  \cite[Theorem 2]{LTT} show
 that if a variety $\mathcal V$ 
has $k+1$ Day terms, that is, $\mathcal V$ 
is $k$-modular, then $\mathcal V$ has  $\leq k^2 - k +1$  Gumm terms
(apparently, this can be slightly improved in the case $k$ even).
A folklore result in the other direction
is given in the following proposition. See also
Propositions \ref{dm} and \ref{numdd} below for related results.

\begin{proposition} \labbel{numd}
If $\mathcal V$ has
$n+2$ Gumm terms, then 
$\mathcal V$ is $2n+2$-modular.
 \end{proposition} 

 \begin{proof} 
This is the case $q=2$ of Theorem \ref{qkmod}(ii).
More directly,
take $R= \beta $ and $S= \alpha \gamma \circ \beta $
in identity \eqref{aga} in  Lemma \ref{appgumm} and then take converses. 
\end{proof} 

To the best of our knowledge, it is not
known whether the above bounds  are optimal (in both directions).
As we mentioned, having $2$ Gumm terms
is equivalent to $2$-modularity, and equivalent to congruence
permutability.  The arguments from \cite{LTT}
can be refined to show that 
 $3$-modularity implies the existence of
 $3$ Gumm terms. 
We do not know whether the converse holds;
however, $3$ Gumm terms imply
$4$-modularity, by  Proposition \ref{numd}.
Actually, $3$ Gumm terms imply 
$(n,n)$-modularity, for every $n \geq 4$, by taking 
$\alpha \gamma $ in place of $\gamma$ in 
\cite[equation (2.4) in Theorem 2.3]{jds}.
Hence some of the above results are likely
to be improved.   
In a sense, already the basis of the  general problems
we are considering
is filled with obscurity! 

In any case, whatever the possible methods of proof,
we can ask about the actual and concrete situations
which might occur.

\begin{problem} \labbel{prob} 
(Possible values of the Day spectrum)
Which functions from $\{ n \in  \mathbb N \mid n \geq 3 \}$ to $ \mathbb N$ 
can be realized as
 $D_{_ \mathcal V} $,
for some congruence modular variety  $\mathcal V$?
\end{problem}   

Notice that, as we mentioned, if a variety $\mathcal V$ is 
congruence modular and $r$-permutable, then  
 $D_{  \mathcal V} (m)  \leq r$, for every $m$.
Hence
 $(m, k)$-modularity has little or no influence, in general, 
on $D_{ \mathcal V} (n) $, for $n< m$
(of course, monotonicity has to be respected).
For example, in the variety of $n$-Boolean algebras \cite[Example 2.8]{S,J} 
we have $D_{  \mathcal V} (m)  = \min \{ m,n+1\} $.
In fact, if $\mathcal {V}$ is the variety of 
$n$-Boolean algebras, then $\mathcal {V}$ has  lattice
operations,
hence $D_{  \mathcal V} (m)  \leq m $,
for every $m \geq 3$, 
by Proposition \ref{small} (i). 
On the other hand, $\mathcal {V}$  is $n + 1$-permutable, thus 
$D_{  \mathcal V} (m)  \leq n+1 $, for every $m$.
If $m \leq n$,
we cannot have 
 $D_{  \mathcal V} (m)  < m $,
since this would imply $m$-permutability, 
by taking $\alpha=1$,
hence we would have 
$n$-permutability, since    
$m \leq n$, but $\mathcal {V}$ is 
 not $n$-permutable.
See \cite[p.\ 3]{jds} 
for corresponding examples in the case of spectra associated
with congruence distributivity.

In the other direction, as we proved in Proposition \ref{thm2}, 
$(m, k)$-modularity puts some tight constraints
on $D_{_ \mathcal V} (n) $, for 
$n > m$. 
Another perspective to appreciate
this  aspect is to use Corollary \ref{qkmod2} and 
Theorem \ref{comb}
via the above-mentioned result by Lakser,  Taylor and  Tschantz  \cite{LTT}.

\section{Connections with distributivity} \labbel{probss}

Notions  and problems corresponding to Problem \ref{prob} 
in the case of congruence distributive varieties
have been studied in \cite{jds}, though
not every problem has been yet completely solved
even in this relatively simpler case. Cf. also \cite{B,ia}.

Of course, every congruence distributive variety is
congruence modular. This can be witnessed at the level
of terms, as already shown by Day \cite[Theorem on p. 172]{D}.
But the problem he asked shortly after about the optimal number of terms
necessary for this is perhaps still open. 

Recall that a variety $\mathcal V$  is \emph{$n+1$-distributive}
if the congruence identity 
$ \alpha ( \beta \circ \gamma  ) \subseteq \alpha \beta 
\circ _{n+1} \alpha \gamma $ holds in $\mathcal V$.
This is witnessed by the existence of $n+2$  J{\'o}nsson terms $j_0, \dots,  j_{n+1}$;
the situation is entirely parallel to Proposition \ref{mod3};
specific details for the case
of congruence distributivity can be found, e.~g., in \cite{jds}. J{\'o}nsson 
terms have been briefly recalled in Remark \ref{dj}(i). For short,  J{\'o}nsson terms
for  $n+1$-distributivity are terms satisfying condition (ii) in Theorem \ref{gummt},
with equation (b)  replaced by  $x=p(x,y,z)$, thus
$p$ can be safely relabeled $j_0$ and (c) becomes an instance of (e).  
J{\'o}nsson \cite{JD} showed that a variety is congruence distributive 
if and only if it is $n$-distributive, for some $n$.  
In  a more general context, in \cite{jds} we introduced the 
following  \emph{J{\'o}nsson 
distributivity function} $ J _{ \mathcal V}$, 
  for a congruence distributive variety $\mathcal V$. 
$ J _{ \mathcal V} (m) $ is the least 
$k$ such that 
$\mathcal V$ satisfies the identity
$\alpha ( \beta \circ _{m+1} \gamma ) \subseteq 
\alpha \beta \circ _{k+1} \alpha \gamma$. Notice the shift by $1$.
   
The mentioned theorem by
Day on \cite[p.\ 172]{D}
states that if a variety $\mathcal V$ is $n+1$-distributive,
then $\mathcal V$ is $(2n+1)$-modular (for notational
convenience, here we are again shifting
 by $1$ with respect to \cite{D}). The theorem 
might be seen as a predecessor to some results
from 
\cite{jds},
since 
the proof of  \cite[Theorem on p.\ 172]{D}
actually shows that if $\mathcal V$ is $n+1$-distributive, 
that is, $ J _{ \mathcal V} (1) =n$, then
$ J _{ \mathcal V} (2) = 2n$,  the  special case 
$\ell= 2$ of
\cite[Corollary 2.2]{jds}.
Indeed, the terms constructed on
 \cite[p.\ 172]{D} satisfy $d_i(x,y,z,x)=x$, for every $i$.
In terms of congruence identities, they thus witness
$\alpha( \beta \circ \gamma \circ \beta ) \subseteq 
\alpha \beta \circ _{2n+1} \alpha \gamma  $,
rather than simply $\alpha( \beta \circ \alpha \gamma \circ \beta ) \subseteq 
\alpha \beta \circ _{2n+1} \alpha \gamma  $.
Conversely, the  identity 
$\alpha( \beta \circ \gamma \circ \beta ) \subseteq 
\alpha \beta \circ _{2n+1} \alpha \gamma  $
can be obtained as a consequence of
\cite[Corollary 2.2]{jds}, which in particular furnishes a proof of Day's result, taking
$\alpha \gamma $ in place of $\gamma$.

We now observe that
when $n$ is even Day's result can be improved by $1$.

\begin{proposition} \labbel{dm}
If $n$ is even and $\mathcal V$ is $n+1$-distributive,
then $\mathcal V$ is  $2n$-modular. 
 \end{proposition}

 \begin{proof}
For the reader's convenience, 
we are going to give an explicit and direct proof below, but first
let us observe that the proposition can be obtained as a consequence of 
Proposition \ref{numd}. Indeed, by 
Remark \ref{dj}(iii) and since $n+1$ is odd, 
we have that
$n+1$-distributivity is equivalent to   
ALVIN $n+1$-distributivity.
Hence, by Remark \ref{dj}(iv),  
terms witnessing ALVIN $n+1$-distributivity 
are in particular a set of $n+1$ Gumm terms.
Then Proposition \ref{numd} (with $n-1$ in place of  $n$) 
gives the result. 

Alternatively, a  direct proof can be obtained as in 
\cite[Theorem on p.\ 172]{D}, with a suitable
variation concerning  the  terms at the end.
Given J{\'o}nsson  terms $j_0, \dots, j_{n+1}$, we obtain the following 
terms $d_0, \dots, d_{2n}$ satisfying condition (iv) in Proposition \ref{mod3}.
The terms  are considered as quaternary terms depending on the variables
$x,y,z,w$. We omit commas for lack of space. 
\begin{align*}   
  d_0&{\,=\hspace {1 pt}}j_0(x y w)   &  d_1&{\,=\hspace {1 pt}}j_1(x y w) &   d_2&{\,=\hspace {1 pt}}j_1(x z w) &     
   d_3&{\,=\hspace {1 pt}}j_2(x z w)
 \\   
 d_4&{\,=\hspace {1 pt}}j_2(x y w) &
   d_5&{\,=\hspace {1 pt}}j_3(x y w)  & 
   d_6&{\,=\hspace {1 pt}}j_3(x z w) & & \dots 
\\
  d_{4i}&{\,=\hspace {1 pt}}j_{2i}(x y w)    & d_{4i{{+}}1}&{\,=\hspace {1 pt}}j_{2i{+}1}(x y w)    &
  d_{4i{{+}}2} &{\,=\hspace {1 pt}}j_{2i{+}1}(x z w)   &
  d_{4i{+}3}& {\,=\hspace {1 pt}}j_{2i{+}2}(x z w) 
\\
& \dots &   d_{2n{-}7} &{\,=\hspace {1 pt}}  j_{n{-}3}(x y w)
&   d_{2n{-}6} &{\,=\hspace {1 pt}}  j_{n{-}3}(x z w)  &
  d_{2n{-}5} &{\,=\hspace {1 pt}}  j_{n{-}2}(x z w)  
\\
 d_{2n{-}4} &{\,=\hspace {1 pt}}  j_{n{-}2}(x y w) &   d_{2n{-}3} &{\,=\hspace {1 pt}}  j_{n{-}1}(x y w)
&   d_{2n{-}2} &{\,=\hspace {1 pt}}  j_{n{-}1}(x z w)  &
  d_{2n{-}1} &{\,=\hspace {1 pt}}  j_{n}(y z w)  
\\
  d_{2n} & {\,=\hspace {1 pt}}  w.  
 \end{align*}

Notice the different argument of $j_{n}$
in comparison with $j_2, \dots, j_{n-2}$. 
Notice that $2n-4 = 4 \frac{n-2}{2}$,
hence the  indices in the penultimate line follow the same pattern
as in the preceding lines, taking $i=  \frac{n-2}{2} $.
We can do this since $n$ is assumed to be even. 

The fact that  $d_0, \dots, d_{2n}$ satisfy Condition (iv) in Proposition \ref{mod3}
is easy and is proved as in \cite[p.\ 172]{D}.
The only differences are that 
$ d_{2n-2} (x,x,w,w)=  j_{n-1}(x,w,w) =   j_{n}(x,w,w) 
=d_{2n-1} (x,x,w,w)$ 
and that 
$d_{2n-1} (x,y,y,w)=  j_{n}(y,y, \allowbreak w) = j_{n+1}(y,y,w)  = w
= d_{2n} (x,y,y,w)$. 
Notice again that it is fundamental to have $n$ even!
 \end{proof}

\begin{remark} \labbel{rmk2}   
We have not used the equation
$j_{n}(x,y,x) = x$ in the above proof.
This is another way to see that 
$n+1$ Gumm terms imply $2n$-modularity, as proved in 
 Proposition \ref{numd}
(here the order of terms and of variables should be reversed).
Put in another way, 
  the congruence identity
$\alpha( \beta \circ \gamma  ) \subseteq 
(\alpha \beta \circ_{n-1} \alpha \gamma ) \circ \alpha ( \gamma \circ \beta )$ 
holding in some variety  implies
$\alpha( \beta \circ \alpha  \gamma \circ \beta ) \subseteq 
\alpha \beta \circ_{2n} \alpha \gamma $. 

Notice however that the argument in 
the proof of Proposition \ref{dm}  \emph{does not} improve 
the value of $ J _{ \mathcal V} (2) $ mentioned before the
statement of the proposition.
We only get from the proof that, for $n$ even, $n+1$-distributivity, 
or just having $n+1$ Gumm terms, 
 imply 
$\alpha( \beta \circ \gamma \circ \beta )
 \subseteq ( \alpha \beta \circ _{2n-2}  \alpha  \gamma ) 
\circ \alpha (\beta \circ \gamma )  $. 
This is because the terms constructed in the proof of 
 \ref{dm} satisfy
$d_h(x,y,z,x)=x$, for 
$h \leq 2n-2$.  
Compare also the case 
$q=1$ in  identity \eqref{qdistconv} in Theorem \ref{agtcor}.  
We believe that the point is best seen in terms of congruence
identities.
 In this sense, the argument in \ref{dm} works for modularity, since 
in this case we have $\alpha \gamma $ in place of $\gamma$, hence
$ \alpha (\beta \circ \alpha \gamma  ) =
 \alpha \beta \circ \alpha \gamma$.  
 \end{remark}

\begin{definition} \labbel{rm}   
Just as it is possible to define
ALVIN or reversed $n$-distributivity, 
we can consider reversed modularity.
In detail, let us say that a variety $\mathcal V$ 
is \emph{$k$-modular$ ^\smallsmile $} 
if
$ \alpha ( \beta \circ \alpha \gamma \circ \beta  ) \subseteq 
\alpha \gamma  \circ_k \alpha \beta $ holds in $\mathcal V$.
Notice that
 if $k$ is even, then,
 by taking converses,
$k$-modularity$ ^\smallsmile $ 
and
$k$-modularity are equivalent notions.
On the other hand, for $k$ odd, we get distinct notions, in general.
For example,  $3$-modularity$ ^\smallsmile $ 
implies $3$-permutability (just take $\alpha=1$), while
$3$-modularity does not (e.~g., consider the variety of lattices).
As a marginal remark, it is well-known that $3$-permutability 
implies congruence modularity, hence obviously it also
 implies  
$3 $-modularity$^\smallsmile$. In particular, we get that
$3$-permutability and   
$3 $-modularity$ ^\smallsmile $
are equivalent notions.
 \end{definition}

Of course, $k$-modularity$ ^\smallsmile $ 
can be given a Maltsev characterization as in Proposition \ref{mod3}(iv),
just exchange odd and even. 

Recall the definition of defective Gumm terms 
given right before Proposition \ref{rmk1}. 
We can improve Proposition \ref{numd}
in case $n$ even if we express the conclusion in terms
of reversed modularity.

\begin{proposition} \labbel{numdd}  
If  $n$  is even and $\mathcal V$ has
$n+2$ defective Gumm terms, then
$\mathcal V$ is $2n+1$-modular$ ^\smallsmile $.
\end{proposition}

 \begin{proof} 
Take $R= \beta $ and $S= \alpha \gamma \circ \beta $ 
in Proposition \ref{rmk1}. 
\end{proof}

Of course, in case $n$ even,
Proposition \ref{numdd} generalizes
Proposition \ref{numd}, since 
$2n+1$-modularity$ ^\smallsmile $
obviously implies $2n+2$-modularity.
By the way, the observation that one can delete 
one equation from  Gumm conditions, still obtaining
a property implying congruence  modularity,
has appeared before in Dent, Kearnes and Szendrei 
\cite[Theorem 3.12]{DKS}
in a different context and
with somewhat different terminology and notations. 

We can use Proposition \ref{numdd}
in order to improve the induction basis  in, say,
Theorem \ref{qkmod} when $n$ is even, thus improving by $1$ or $2$  the
evaluations of  $D_{ \mathcal V} (m) $ which follow 
from \ref{qkmod}. Some computations can be expressed in terms of the reversed
spectrum   $D_{ \mathcal V} ^\smallsmile (m) $ which 
we shall define at the beginning of the next section.

\section{Generalized spectra} \labbel{probsss}  

We now notice that $D_{ \mathcal V} $ is not the only spectrum which deserves consideration. Actually, it appears quite natural 
to introduce a great deal of spectra, as we are going to show.
First, along the above lines, we can define also reversed 
$(m,k)$-modularity.
We 
  define a variety to be  
\emph{$(m,k) $-modular$^\smallsmile $}
if 
$ \alpha ( \beta \circ_m \alpha \gamma) \subseteq 
\alpha \gamma  \circ_k \alpha \beta $
holds in $\mathcal V$.
Thus $(3,k) $-modularity$^\smallsmile $
is what we have called $k$-modularity$^\smallsmile $.
Of course, if $\mathcal V$ is 
 $(m,k) $-modular$^\smallsmile $, then
$\mathcal V$ is $(m,k+1) $-modular
and, similarly, if 
$\mathcal V$ is $(m,k) $-modular, then
$\mathcal V$ is 
$(m,k+1) $-modular$^\smallsmile $.
Put in another way, if we choose the best possible values
for $k$ 
both in 
 $(m,k) $-modularity and in
$(m,k) $-modularity$^\smallsmile $,
these optimal values differ at most by $1$.
Notice also that, exactly as in the case $m=3$,  if $m$ is odd and $k$ is even, then
 $(m,k) $-modularity is equivalent to 
$(m,k) $-modularity$^\smallsmile $:
just take converses.
On the other hand, it follows from the special case
$m=3$ treated in the comment in Definition \ref{rm}  that
 $(m,k) $-modularity and 
$(m,k) $-modularity$ ^\smallsmile $
are distinct notions, in general.

In particular, it is interesting to study
the \emph{reversed Day spectrum} 
 $D ^\smallsmile _{ \mathcal V} (m) $  
of a congruence modular variety $\mathcal V$. We let 
$D ^\smallsmile _{ \mathcal V} (m) $
 be the smallest $k$ such that 
$\mathcal V$ is $(m,k) $-modular$  ^\smallsmile $.
By the above comments, 
 $D_{ \mathcal V} $ and  $D ^\smallsmile _{ \mathcal V} $  
differ at most by $1$, but are different functions, in general. 

It is probably interesting to consider also the following
variant of $D_{ \mathcal V}$.  
Let $D^*_{ \mathcal V} (\ell) $  
be the least $k$ such that $\mathcal V$ satisfies 
\begin{equation} \labbel{A}   
\alpha ( \beta \circ \alpha ( \gamma \circ \alpha ( \beta \circ \dots
\alpha ( \gamma  ^ \bullet  \circ
 \alpha ( \beta^ \bullet \circ \alpha \gamma ^ \bullet  \circ  \beta^ \bullet )
 \circ \gamma ^ \bullet ) \ldots
 \circ \beta )\circ \gamma) \circ \beta ) \subseteq
\alpha \beta \circ_{k} \alpha \gamma  
\end{equation}
with exactly $\ell$ open brackets
 and exactly $\ell$ closed brackets, and where 
$\gamma^ \bullet = \gamma  $ and 
$\beta^ \bullet = \beta  $ if $\ell$ is odd, and 
$\gamma^ \bullet = \beta  $ and 
$\beta^ \bullet = \gamma  $ if $\ell$ is even.

Notice that  $D^*_{ \mathcal V} (1)  = D_{ \mathcal V} (3) $.  
By an easy induction, using Theorem \ref{thm},
we see that in a congruence modular variety,
for every $\ell$, there is some $k$ such that 
\eqref{A} holds.
The next proposition provides a much better bound for
$D^*_{ \mathcal V}$,
with respect to the bound which can be obtained using
Theorem \ref{thm}.

\begin{proposition} \labbel{dst}
If $\mathcal V$ is $(3, k)$-modular,
with $k$ even, say, $k=2r$, then
$D^*_{ \mathcal V} ( \ell) \leq 
\frac{k ^ \ell}{2 ^{ \ell-1} } = 2 r ^{ \ell}  $,
for every $\ell>0$.    
 \end{proposition}  

\begin{proof}
The proof is by induction on $\ell$. 
By the definitions, 
$D^*_{ \mathcal V} (1) 
= D_{ \mathcal V} (3) \leq k = 2r$,
thus we get the basis of the induction $\ell=1$.

We now prove  the induction step.
Let us denote by 
$A_ \ell $ the expression on the left-hand side of
 \eqref{A}
and by $A_ \ell ^ \bullet$
the similar expression 
in which all the occurrences of  $\beta$ and $\gamma$ are exchanged.
Suppose that we have proved the proposition 
for some $\ell$
and let $h=2r ^ \ell  $. By exchanging the role
of $\beta$ and $\gamma$, we thus get 
$A_ \ell ^ \bullet \subseteq
\alpha \gamma  \circ_{h} \alpha \beta  $
 and, by taking converses and since $h$ is even,
we also get 
$A_ \ell ^ \bullet \subseteq
\alpha \beta  \circ_{h} \alpha \gamma    $. 

\begin{claim} \labbel{claim}
$\alpha ( \beta \circ \alpha A_ \ell ^\bullet \circ \beta ) 
\subseteq 
 \alpha \beta   \circ_ \kappa \alpha  A_ \ell ^\bullet $ 
  \end{claim}  

  \begin{proof}[Proof of the Claim] 
Were $A_ \ell ^\bullet $ a \emph{representable}
tolerance, the claim would immediately follow
from   
  \cite[Theorem 3 (i) $\Rightarrow $  (iii)]{L}.
Cf.\ 
Remark \ref{czh}.
However, though formally the Claim does not follow
from \cite{L}, we show that essentially the same arguments
can be carried over. In fact, some arguments from \cite{L}
can be generalized and applied to a larger class of tolerances,
which we call \emph{nest-representable} in \cite{nest}.
We refer to   \cite{nest} for details about the general proof,
while here we limit ourselves  to exemplify
 the method for the specific case at hand.

If $(a, d) \in \alpha ( \beta \circ \alpha A_ \ell ^\bullet \circ \beta )  $,
then $a \mathrel {\alpha} d$
and there are elements $b$ and $c$
such that 
$a \mathrel \beta b \mathrel {  \alpha A_ \ell ^\bullet } c \mathrel  \beta d$.
Iterating,    
we get   elements 
$ b=b_0, b_1, \dots, b_ \ell $ and 
$ c=c_0, c_1, \dots, c_\ell$
such that 
\begin{equation}\labbel{bc}
\begin{aligned}   
&& &b_j \mathrel \alpha c_j, &&\text{for  }j=0,\dots, \ell, \\
&& &b_\ell \mathrel {\beta ^\bullet} c_\ell, \\
&b_j  \mathrel \gamma  b_{j+1},  & &c_j \mathrel \gamma  c_{j+1}, && \text{for } j \text{ even, } 0\leq j\leq \ell -1, \\
&b_j \mathrel \beta  b_{j+1},  & &c_j \mathrel \beta  c_{j+1}, && \text{for } j \text{ odd, } 0\leq j\leq \ell -1. 
\end {aligned} 
 \end{equation}    

Cf.\ \cite[Section 2]{abh} for a similar situation and for a diagram representing
it. Let $d_0, \dots, d_k$ be Day terms given by $(3,k)$-modularity 
and Proposition \ref{mod3}(iv).  
Consider the elements 
$d_i(a,b,c,d)$, for $i=0, \dots, k$.
It is by now standard to see that 
$d_i(a,b,c,d) \mathrel \alpha  d _{i+1}(a,b,c,d) $, for $i=0, \dots, k-1$
and that 
$d_i(a,b,c,d) \mathrel \beta   d _{i+1}(a,b,c,d) $, for even $i < k$.
The Claim will follow if in addition we prove that
$d_i(a,b,c,d) \mathrel { \alpha A_ \ell ^\bullet}   d _{i+1}(a,b,c,d) $, for odd  $i < k$.

Fix some odd  $i < k$ and consider the elements
$b'_j=d_i(a,b_j,c_j,d)$ and 
$c'_j=d_{i+1}(a,b_j,c_j,d)$,
for $j=0, \dots, \ell $.  
The conditions in \eqref{bc}, together with the equation
 $d_i(x,y,y,w)=   d _{i+1}(x,y,y,w) $
entail that the elements $b'_j$ and  $c'_j$
satisfy the conditions in \eqref{bc}, too, hence 
$d_i(a,b,c,d) \mathrel { \alpha A_ \ell ^\bullet}   d _{i+1}(a,b,c,d) $
and the Claim follows.
\qedhere$_ {Claim}$  \end{proof}

 \emph{Proof of  Proposition   \ref{dst} (continued)} 
We finally compute
$A _{\ell+1} = \alpha ( \beta \circ \alpha A_ \ell ^\bullet \circ \beta ) 
\subseteq ^{\text{Claim}} 
 \alpha \beta   \circ_ \kappa \alpha  A_ \ell ^\bullet \subseteq ^{\text{ih}}
 \alpha \beta   \circ_ \kappa ( \alpha \beta \circ _h \alpha \gamma )
=  ( \alpha \beta \circ _h \alpha \gamma )  ^{ \frac{k}{2}} 
 =   \alpha \beta \circ _ { \frac{hk}{2}} \alpha \gamma  $,
since  $h$  is even, thus we have completed the induction step, since
$ \frac{hk}{2}= \frac{2r ^{\ell} \cdot 2r }{2} = 2r ^{ \ell +1}  $. 
 \end{proof} 

As another invariant,  inspired by
Tschantz \cite{T},
we can define the
 \emph{Tschantz modularity function} $T_ { \mathcal V}$ 
for a congruence modular variety $\mathcal V$ 
in such a way that, for
 $m \geq 2$,    $T _ { \mathcal V}(m)$ is the least $k$ 
such that the following congruence identity holds in $\mathcal V$
\begin{equation*}\label{Teq}
 \alpha ( \beta \circ_{m} \gamma )
\subseteq 
\alpha (\gamma \circ \beta ) \circ (\alpha \gamma  \circ _{k} \alpha \beta ). 
  \end{equation*}    

Notice that 
$T _ { \mathcal V}(2)=0$
is equivalent to congruence permutability,
just take $\alpha=1 $  
(recall that, by convention we let $ \beta  \circ_0 \gamma  = 0$). More
generally, $T _ { \mathcal V}(2)=k$
if and only if $\mathcal V$ has 
$k+2$ Gumm terms  
$m, j_1, \dots, j_{k+1}$,
but   not 
$k+1$ Gumm terms. 
 The definition of $T_ { \mathcal V}$ 
makes sense for a congruence modular variety, 
as implicit from \cite{T}.  More explicitly, the existence of 
 $T _ { \mathcal V}(m)$ 
follows from identity \eqref{qdist} in  Theorem  \ref{agtcor}. 
Further constraints on  $T _ { \mathcal V}(m)$  are provided
by Corollary \ref{agtcor2}.  A rather surprising result,
a reformulation of parts of \cite[Theorem 2.3]{jds},
asserts that if $\mathcal V$ has $3$ Gumm terms, that is,
 if  $T _ { \mathcal V}(2) \leq 1$,
then  $T _ { \mathcal V}(m+2) \leq m$,
for $m \geq 2$. See, in particular, 
\cite[equation (5)]{jds}. 

Of course, we can also define the \emph{reverse Tschantz
function}
 $T_ { \mathcal V} ^\smallsmile $ 
as the least $k$ 
such that the following congruence identity holds in $\mathcal V$
\begin{equation*}\label{Teqrev}
 \alpha ( \beta \circ_{m} \gamma )
\subseteq 
\alpha (\beta \circ \gamma ) \circ (\alpha \beta   \circ _{k} \alpha \gamma  ). 
  \end{equation*}    
Notice that, for $m=2$,  the above identity 
is satisfied by \emph{every} variety (with $k=0$),
hence in the case of the   reverse Tschantz
function it is appropriate to start with $m=3$. 

The results from  \cite{T} can be used in order to introduce more spectra.
For example, for $m \geq 3$, we can let $T^* _ { \mathcal V}(m)$ to be the least $k$ 
such that 
$ \alpha ( \beta \circ_{m} \gamma )
\subseteq 
\alpha (\gamma \circ \beta \circ \gamma  ) \circ 
(\alpha \beta  \circ _{k} \alpha \gamma  ) 
$. Similarly, 
we can let $T^{**} _ { \mathcal V}(m)$ to be the least $k$ 
such that 
$ \alpha ( \beta \circ_{m} \gamma )
\subseteq 
\alpha (\gamma \circ \beta \circ \gamma  ) \circ_k 
\alpha (\gamma \circ \beta \circ \gamma  )$.
Notice that the above identities  do imply
congruence modularity, just consider $\alpha \gamma $
in place of $\gamma$. More identities equivalent to congruence modularity
have been introduced in \cite{ricmc,ricm}.

The considerations in Remark \ref{agaimprem}, in particular,
identity \eqref{bbb},  suggest 
to consider also the following
ternary relation  $T^{***} _ { \mathcal V}(m,h,k)$, defined
to hold when
$\mathcal V$ satisfies 
$ \alpha ( \beta \circ_{m} \gamma )
\subseteq 
(\alpha \beta  \circ _{h} \alpha \gamma  ) \circ 
\alpha (\gamma \circ \beta  ) \circ 
(\alpha \gamma   \circ _{k} \alpha \beta ) 
$.
As above, taking $\alpha \gamma $ in place of $\gamma$,
we see that if $T^{***} _ { \mathcal V}(m,h,k)$ holds and
$m \geq 3$, then $\mathcal V$ is congruence modular.
Similarly, 
we can consider the satisfaction of identities
of the form $ \alpha ( \beta \circ_{m} \gamma )
\subseteq 
(\alpha \beta  \circ _{h} \alpha \gamma  ) \circ 
\alpha (  \gamma \circ \beta  \circ \gamma ) \circ 
(\alpha \gamma   \circ _{k} \alpha \beta ) 
$.

Quite surprisingly, results like Tschantz'
hold even when $\beta$ and $\gamma$ are replaced by 
 admissible relations; this relies heavily
on  Kazda, Kozik,  McKenzie, Moore \cite{adjt};
see \cite{ricmc,ricm} for details.
Hence we can also define the   
\emph{relational Tschantz
functions}
 $T_ { \mathcal V} ^{r} (m)$
and
 $T_ { \mathcal V} ^{r \smallsmile} (m)$
as the least $k$'s  
such that, respectively,  the following congruence identities hold in $\mathcal V$
\begin{equation*}\label{Teqr}
 \alpha ( R \circ_{m} S )
\subseteq 
\alpha (S \circ R ) \circ (\alpha S   \circ _{k} \alpha R  ) 
  \end{equation*}    
\begin{equation*}\label{Teqrevr}
 \alpha ( R \circ_{m} S )
\subseteq 
\alpha (R \circ S ) \circ (\alpha R   \circ _{k} \alpha S  ) 
  \end{equation*}    

Constraints on $T_ { \mathcal V} ^{r \smallsmile } (m)$
are provided, for example, by 
\cite[equation (6) in Corollary 2.2]{ricm}. 
Since here $R$ and $S$ are not assumed to be transitive,
we can consider also identities like 
$\alpha ( R \circ_{m} S )
\subseteq 
\alpha (S \circ R ) \circ (\alpha S   \circ  \alpha R
\circ \alpha R \circ \alpha S \circ \alpha S \circ \alpha R \circ \dots  ) 
$, or even
$\alpha ( R_1 \circ_{m} R_2 )
\subseteq 
\alpha ( R_2 \circ_{m} R_1) \circ (\alpha   R_{f(1)} \circ 
\alpha R_{f(2)} \circ \dots ) 
$,
where $f$ is a function with codomain $\{ 1,2 \}$.
Some results and considerations in \cite{uar,ia}   suggest that it is natural to deal 
with the above identities.

Still other characterizations of
congruence modular varieties are possible.
In \cite{ricmc,ricm} we have proved that a variety is congruence modular 
if and only if, for every $m \geq 2$
(equivalently, for $m = 2$), there is some $k$ such that   
\begin{equation}\labbel{R}
 \alpha ( R \circ_{m} R )
\subseteq 
\alpha R   \circ _{k} \alpha R 
  \end{equation}    
Thus we can ask about the possible values of the function
$R_{ \mathcal V} $ assigning to $m$
the smallest possible value $k$  such that   
\eqref{R} holds in $\mathcal V$. 

By replacing $R$ with $R \circ  \alpha S$ in \eqref{R},
it easily follows that  
a variety is congruence modular 
if and only if, for every $m \geq 3$
(equivalently, just for $m=3$,
equivalently, for some $m \geq 3$), there is some $k$ such that   
the relational version of modularity
\begin{equation}\labbel{RR}
 \alpha ( R \circ_{m} \alpha S )
\subseteq 
\alpha R   \circ _{k} \alpha S
  \end{equation}    
 holds in $\mathcal V$.
We let $D^r_{ \mathcal V} (m) $  
be the least $k$ such that  
\eqref{RR} holds; 
$D^{r \smallsmile }_{ \mathcal V} (m) $ is defined as usual.

Finally, we shall consider an identity dealing with tolerances
\begin{equation}\labbel{tol}     
\Theta^h \Psi^k  \subseteq (\Theta \Psi) ^{\ell},  
  \end{equation}
where $\Theta^h $ is a shorthand for 
$\Theta \circ_ h \Theta  $.
This time we get a function $C_{ \mathcal V} $  depending on two arguments:
$C_{ \mathcal V}(h,k) $ is the least $\ell$ 
  such that \eqref{tol} holds.
The definition  of $C_{ \mathcal V} $
is justified since congruence modular varieties satisfy
the tolerance identity $\Theta^* \Psi^*  = (\Theta \Psi) ^*$,
  usually called TIP. See, e.~g., Cz\'edli, Horv{\'a}th and Lipparini \cite{CHL}
for some history, references and applications. 

It seems to be an open problem
whether we still get identities equivalent to congruence modularity
if we let \emph{both}
 $\Theta$  and $\Psi$ vary among  admissible relations 
in the identity \eqref{tol}. 
In other words, we do not know whether 
every congruence modular variety satisfies
the identity 
$R^* S^*  = (RS) ^*$.
On the other hand, we showed in \cite{ricmc,ricm} 
that we can let $\alpha$ be a tolerance
in equations \eqref{R} and \eqref{RR},
possibly with a different  value of $k$. 
A characterization of varieties satisfying 
$R^* S^*  = (RS) ^*$
appears in \cite[Proposition 1]{ricmc}.

\begin{problem} \labbel{probb} 
(Generalized  Day spectra)
Which functions 
can be realized as
$T_ { \mathcal V} $,
for some congruence modular variety  $\mathcal V$?
The same problem for all the spectra introduced above.

More generally and globally, which
13-tuples of functions can be realized as 
$(  D  _{ \mathcal V}, \allowbreak   D ^\smallsmile _{ \mathcal V} ,  D ^* _{ \mathcal V},
T_ { \mathcal V} , 
T_ { \mathcal V} ^\smallsmile,
T^{*} _ { \mathcal V},
T^{**} _ { \mathcal V},  \allowbreak 
 T_ { \mathcal V} ^{r  } ,
 T_ { \mathcal V} ^{r \smallsmile },
R_ { \mathcal V} ,
D^{r  }_{ \mathcal V}  ,
D^{r \smallsmile }_{ \mathcal V},
C_ { \mathcal V}  )$,
for some congruence modular variety  $\mathcal V$?
\end{problem}

In the special case of a congruence distributive variety 
(which is, in particular,  congruence modular) it is probably interesting
to consider simultaneously the spectra introduced 
in \ref{probb} and those introduced in \cite{jds}. 
Some connections between the J{\'o}nsson 
and the modularity spectra have been presented in 
\cite[Section 4]{jds}, together with some further problems. 
A large number of the above spectra have been evaluated
  \cite{B} in the case of  Baker's variety,
 the variety generated by polynomial reducts of lattices
in which only the term $x(y+z)$ is taken as a basic operation.   
The note \cite{B} suggests that it is interesting to study identities
of the form  $\alpha ( \beta \circ ( \alpha \gamma   \circ \alpha \beta  
 \circ {\stackrel{2h+1}{\dots}} \circ \alpha \gamma ) \circ \beta )
\subseteq
\alpha \beta \circ _k \alpha \gamma $.

We finally show that all the spectra under consideration
are closed under taking pointwise maximum.

\begin{proposition} \labbel{mal}
For each of the spectra listed in Problem \ref{probb}, 
the set of realizable functions is closed under pointwise maximum.
In detail, if $\mathcal V$ and $\mathcal V'$  are congruence modular varieties,
then there is a congruence modular variety $\mathcal W$ such that,
for every $m \geq 3$,  
$D_{ \mathcal W} (m)  = \max
 \{ D_{ \mathcal V} (m) , \allowbreak  D_{ \mathcal V'} (m)  \}$,  
and the same holds for all the other spectra.    

The same holds globally, too; namely, 
for every pair of realizable 13-uples  as in Problem \ref{probb},
there is a congruence modular variety realizing their
componentwise maximum. 
 \end{proposition}  

\begin{proof} 
This is the same argument as
in \cite[Proposition 1.1]{jds}.
The non-indexed product of two varieties satisfies
exactly the same strong Maltsev conditions satisfied by both varieties,
and each condition of the form $D_{ \mathcal V} (m)  \leq k$
is a strong Maltsev condition, and the same for all the other spectra. 
This is a consequence of the algorithm by   A.\ Pixley \cite{P}
and R.\ Wille
\cite{W}. 
The cases of relational spectra and of 
 $C_ \mathcal {V}$ are slightly nonstandard, but the proofs present
no essential difficulty. Compare, e.\ g., \cite[Proposition 3.8]{uar}
and  \cite[Lemma 2.3]{ia}.  
\end{proof}

\begin{acknowledgement} \labbel{ack}
We thank G.\ Cz\'edli   for stimulating  correspondence.
 We thank the students of Tor Vergata
University for  stimulating discussions.
 \end{acknowledgement}

\medskip

{\scriptsize
Though the author has done his best efforts to compile the following
list of references in the most accurate way,
 he acknowledges that the list might 
turn out to be incomplete
or partially inaccurate, possibly for reasons not depending on him.
It is not intended that each work in the list
has given equally significant contributions to the discipline.
Henceforth the author disagrees with the use of the list
(even in aggregate forms in combination with similar lists)
in order to determine rankings or other indicators of, e.~g., journals, individuals or
institutions. In particular, the author 
 considers that it is highly  inappropriate, 
and strongly discourages, the use 
(even in partial, preliminary or auxiliary forms)
of indicators extracted from the list in decisions about individuals (especially, job opportunities, career progressions etc.), attributions of funds, and selections or evaluations of research projects.
\par
}

\end{document}